\documentclass[11pt,a4paper]{amsart}
\usepackage[left=2cm,right=2cm,top=3cm,bottom=3cm]{geometry}
\usepackage[T1]{fontenc}
\usepackage{lmodern}
\usepackage{amsmath,amssymb,amsthm}
\numberwithin{equation}{section}
\usepackage{graphicx}
\usepackage{float}
\usepackage{tikz}
\usepackage{booktabs}
\usepackage{xurl}
\usepackage[hidelinks]{hyperref}
\usepackage{microtype}

\makeatletter

\def\l@section{\@tocline{1}{4pt}{0pt}{0pt}{\bfseries}}
\def\l@subsection{\@tocline{2}{0pt}{1.5em}{0pt}{\normalfont}}
\renewcommand{\tocsection}[3]{%
  \@ifnotempty{#2}{%
    \@ifempty{#1}{\makebox[2em][l]{#2.}}{#1~#2.\quad}}#3}

\def\@tocline#1#2#3#4#5#6#7{\relax
  \ifnum #1>\c@tocdepth
  \else
    \par\addpenalty\@secpenalty\addvspace{#2}%
    \begingroup\hyphenpenalty\@M
    \@ifempty{#4}{%
      \@tempdima\csname r@tocindent\number#1\endcsname\relax
    }{\@tempdima#4\relax}%
    \parindent\z@\leftskip#3\relax
    \advance\leftskip\@tempdima\relax
    \rightskip\@pnumwidth plus4em\parfillskip-\@pnumwidth
    #5\leavevmode\hskip-\@tempdima#6\nobreak
    \leaders\hbox{\normalfont\kern.3em.\kern.3em}\hfill
    \nobreak\hbox to\@pnumwidth{\@tocpagenum{#7}}\par
    \nobreak\endgroup
  \fi}
\makeatother

\newtheorem{theorem}{Theorem}[section]
\newtheorem{lemma}[theorem]{Lemma}
\newtheorem{proposition}[theorem]{Proposition}
\newtheorem{corollary}[theorem]{Corollary}
\theoremstyle{definition}
\newtheorem{definition}[theorem]{Definition}
\newtheorem{example}[theorem]{Example}

\title[Solutions in Cantor sets]{Solutions of $x^\alpha+y^\alpha=1$ in $\lambda$-Cantor sets}
\author{Kan Jiang}
\address{Department of Mathematics, Ningbo University, Ningbo 315211, China}
\email{jiangkan@nbu.edu.cn}
\author{Nero Ziyu Li}
\address{Department of Mathematics, Imperial College London, London SW7 2AZ,
United Kingdom; CDT Random Systems, University of Oxford, Oxford, United Kingdom}
\email{ziyu.li21@imperial.ac.uk}
\author{Lifeng Xi}
\address{Department of Mathematics, Ningbo University, Ningbo 315211, China}
\email{xilifeng@nbu.edu.cn}
\date{}

\begin{document}

\begin{abstract}
In this paper, we study the structure of the solutions of \(x^\alpha+y^\alpha=1\) with both
coordinates in a \(\lambda\)-Cantor set.
We prove that, for any \(\alpha>1\), the intersection with the middle-third Cantor
set has positive Hausdorff dimension, strengthening the infinitude question
posed by Jiang et al. [J. London Math. Soc. 113 (2026), e70408].
To explore the structure of the solutions, for arbitrary \(\alpha\), we find
that their Hausdorff dimension undergoes phase transitions as \(\lambda\)
varies: we establish the occurrence of Hausdorff phase transitions by proving the
existence of zero-dimensional and positive-dimensional regions.
\end{abstract}
\maketitle
\tableofcontents

\section{Introduction}\label{sec:introduction}

Does the equation
\[
x^2+y^2=1, \quad x,y \in C ,
\]
where $C$ is the middle-third Cantor set,
have finitely or infinitely many solutions?

The question has a direct geometric interpretation.
Writing \(C\) for the middle-third Cantor set and \(S^1\) for the unit
circle, the solution set is the intersection of the Cantor dust
\(C\times C\) with \(S^1\), as illustrated in Figure~\ref{fig:source}.
The two axis points \((0,1)\) and \((1,0)\) already lie in this
intersection; the question is how many further intersection points
there are, and whether their number is infinite.

\begin{figure}[htb]
\centering
\input{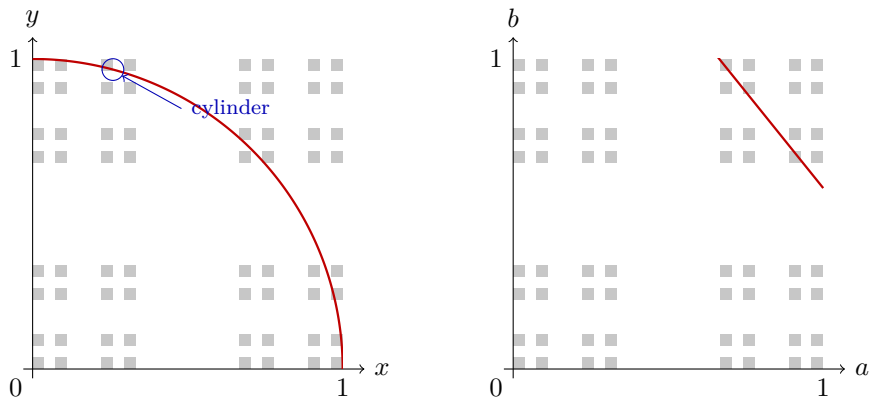}
\caption{The Cantor dust \(C\times C\) and the unit circle \(S^1\)
(left), with a local cylinder shown on the right.}
\label{fig:source}
\end{figure}

A direct antecedent is Yu's study of missing-digit points near
manifolds~\cite{Yu2023}. He asked whether the unit circle contains
infinitely many points whose coordinates lie in the Cantor set of
contraction ratio \(1/5\), and also raised the problem of solving
\(x^3+y^3=1\) in the middle-third Cantor set. These questions concern
prescribed nonlinear curves. The intersection methods developed by
Shmerkin~\cite{Shmerkin2019} and Wu~\cite{Wu2019} provide relevant
dimension estimates, but, as Yu observed, even a zero-dimensional
upper bound does not establish finiteness. For the middle-third Cantor
set, Du, Jiang and Yao~\cite{DuJiangYao2025} subsequently proved, with
computer assistance, that \(x^2+y^2=1\) has at least \(10,000,000\)
distinct solutions with an irrational coordinate. This gave a large
finite lower bound without settling infinitude.

Jiang, Kong, Li and Wang~\cite{Jiang2026} then studied the unit-circle
intersection across the family of \(\lambda\)-Cantor sets. They proved
that only the two axis points occur for \(0<\lambda\leq2-\sqrt3\),
answering Yu's \(1/5\) question negatively, and gave sufficient ranges
for nontrivial intersections and continuum cardinality. They left
infinitude at \(\lambda=1/3\) open and asked for the Hausdorff dimension,
conjecturing \(2\log2/\log3-1\)~\cite[Question~5.1]{Jiang2026}.
More recently, Li, Shen and Wu~\cite{LiShenWu2026} established continuum
cardinality for \((\sqrt3-1)/2<\lambda<1/2\) and positive Hausdorff
dimension for \(\sqrt2-1<\lambda<1/2\). Their paper also includes an
infinitude theorem covering \(\lambda=1/3\), although their
positive-dimension range does not include this parameter. Our
unit-circle result strengthens the known infinitude at \(1/3\) to
positive Hausdorff dimension; the exact dimension remains open.
See also~\cite{HuntKanYorke1993,FengHuangRao2014,HuangKong2023}
for related results on Cantor-set intersections,
\cite{JiangNonlinear,McDonaldTaylor,MoreiraImages} on nonlinear images
of products of Cantor sets, and
\cite{PeresShmerkin2009,NazarovPeresShmerkin} on resonance and the
dimensions of arithmetic sums and convolutions.

Motivated by this question, we seek both to answer it and to understand
the phenomenon in a broader setting.
We replace the exponent \(2\) by a real parameter \(\alpha>1\) and the
middle-third Cantor set by a \(\lambda\)-Cantor set, and study the
solutions of \(x^\alpha+y^\alpha=1\) as these two parameters vary.

%%%%%%%%%%%%%%%%%%%%%%%%%%%%%%%%%%%%%%%%%%%%%%%%%%%%%%%
\subsection{Main Results}\label{subsec:main-results}

The question posed in~\cite{Jiang2026} asks about both the cardinality
and the dimension of the unit-circle intersection in the middle-third
Cantor dust. To understand how the solution set depends on the exponent
and the Cantor set, we ask two questions:
\begin{itemize}
\item (Q1) What happens if we replace the exponent \(2\) by \(\alpha\), even
allowing \(\alpha\) to be real? This will tell us whether the structure of
the solution set depends on the arithmetic structure of the integers.
\item (Q2) If we consider the more general \(\lambda\)-Cantor sets, when
is the solution set zero-dimensional, and when does it have positive
dimension? What do the transitions between these regimes look like
(the Hausdorff phase transitions discussed below)?
\end{itemize}

To answer (Q1), for \(0<\lambda\leq1/2\), we write
\[
C_\lambda:=
\left\{(1-\lambda)\sum_{j=1}^{\infty}
       \varepsilon_j\lambda^{j-1}:\varepsilon_j\in\{0,1\}\right\}.
\]
For \(\alpha>1\), define
\[
E_{\alpha,\lambda}:=
\{(x,y)\in C_\lambda\times C_\lambda:x^\alpha+y^\alpha=1\}.
\]

For \(\lambda=1/3\), we prove in Section~\ref{subsec:classical}
that the solution set has positive Hausdorff dimension for every
real exponent \(\alpha>1\).

\begin{theorem}\label{intro:integer-powers}
For every real \(\alpha>1\), \(\dim_{\mathrm H}(E_{\alpha,1/3})>0\).
\end{theorem}

Taking \(\alpha=2\) strengthens the known infinitude of the unit-circle
intersection to positive Hausdorff dimension. Section~\ref{sec:special}
gives a direct proof of this case with a quantitative lower bound.

Section~\ref{subsec:irrational} proves that, for every real exponent
\(\alpha>1\), the dimension is zero when \(0<\lambda\leq1/4\).
For every irrational exponent, each nonempty open interval in
\((1/4,1/2)\) contains a set of positive Lebesgue measure of
parameters for which the dimension is positive. Thus, for irrational
exponents, positive dimension occurs arbitrarily close to \(1/4\)
from above.
For \(\alpha>1\), we describe the first and last transition parameters by
\[
\begin{aligned}
\lambda_{\min}(\alpha)&:=
\inf\{\lambda\in(0,1/2]:\dim_{\mathrm H}(E_{\alpha,\lambda})>0\},\\
\lambda_{\max}(\alpha)&:=
\sup\{\lambda\in(0,1/2]:\dim_{\mathrm H}(E_{\alpha,\lambda})=0\}.
\end{aligned}
\]

We refer to this change from zero to positive Hausdorff dimension as
a \emph{Hausdorff phase transition}.
We call \(\lambda_{\min}\) and \(\lambda_{\max}\) the first and last
Hausdorff phase transition parameters, respectively.
In familiar models of statistical physics, we usually have a single
phase transition point: one phase throughout the parameter range
below it and a different phase throughout the range above it.
In more complicated models, however, there may be several phase
transition points, or even infinitely many. It is therefore natural
to ask whether, for \(\alpha=2\), there are multiple phase transition
points as \(\lambda\) varies.

The first transition satisfies the following bound.
\begin{theorem}\label{intro:irrational-onset}
For every real \(\alpha>1\), \(\lambda_{\min}(\alpha)\geq1/4\),
with equality whenever \(\alpha\) is irrational.
\end{theorem}

This theorem gives an initial picture of the solution sets across
the parameter family. It suggests considering a Hausdorff phase
transition diagram, with \(\alpha\) on the horizontal axis and \(\lambda\)
on the vertical axis. The diagram records how the two parameters
together determine the changes between zero and positive dimension.

Section~\ref{sec:last} develops this diagram and addresses (Q2).
We obtain an upper bound for the last Hausdorff phase transition
that holds for every real exponent \(\alpha>1\).
Above this common bound, the intersection has positive Hausdorff
dimension throughout the remaining contraction interval.

\begin{theorem}\label{intro:last-real}
For every real \(\alpha>1\),
\[
\lambda_{\max}(\alpha)<\frac{\sqrt3-1}{2}.
\]
\end{theorem}

Beyond positivity, we would like to determine the dimension itself.
The expected value in the positive-dimensional regime is
$
\frac{2\log2}{\log(1/\lambda)}-1,
$
the dimension of the Cantor dust minus one. At \(\lambda=1/3\),
this reduces to \(2\log2/\log3-1\), the value conjectured for the
unit-circle intersection by Jiang, Kong, Li and Wang
in~\cite[Question~5.1]{Jiang2026}.
Determining the dimension at every point of the diagram remains
difficult. Nevertheless, we obtain this formula for almost every
contraction ratio in an explicit interval, for each fixed irrational
exponent.

\begin{theorem}\label{intro:ae-dimension}
For every fixed irrational \(\alpha>1\), for Lebesgue almost every
\(\lambda\in(\sqrt3/4,1/2)\),
\[
\dim_{\mathrm H}(E_{\alpha,\lambda})
=\frac{2\log2}{\log(1/\lambda)}-1.
\]
\end{theorem}

This establishes the expected formula in the stated parameter range;
the precise dimension of the original unit-circle intersection in
the middle-third Cantor dust remains open.

These results provide a framework for studying intersections of the
curves \(x^\alpha+y^\alpha=1\) with \(\lambda\)-Cantor dust and their Hausdorff
phase transitions. Figure~\ref{fig:phase-diagram} summarises the
proved results and numerical estimates in a phase diagram.

\begin{figure}[H]
\centering
\includegraphics[width=\textwidth]{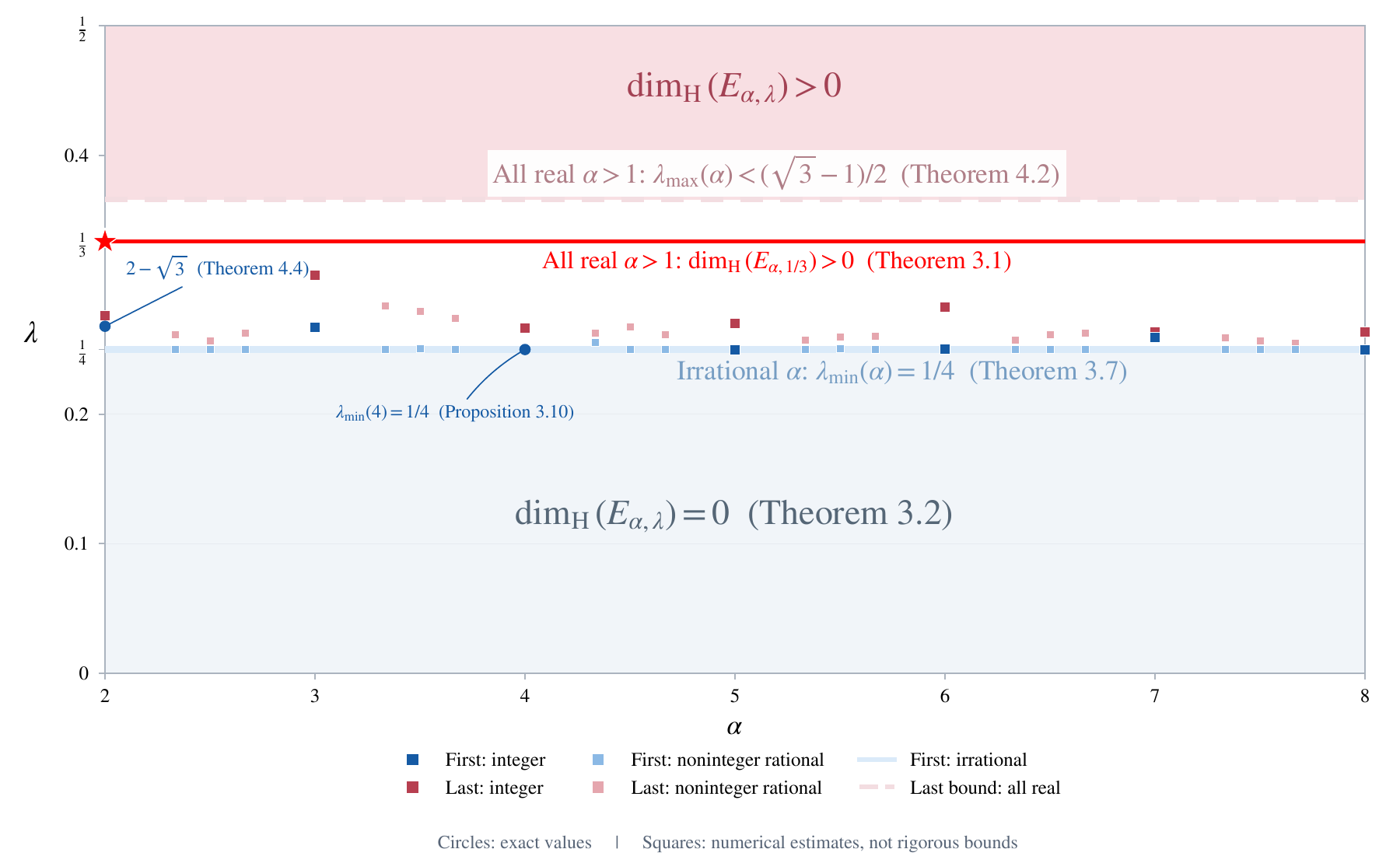}
\caption{Hausdorff phase transition diagram}
\label{fig:phase-diagram}
\end{figure}

\section{\texorpdfstring{\(x^2+y^2=1\)}{x2 + y2 = 1} in the middle-third Cantor set}\label{sec:special}

Our goal in this section is to prove that \(\dim_{\mathrm H}(E_{2,1/3})>0\).

We first locate a nonempty intersection in a pair of Cantor cylinders
and use returns between smaller cylinders to construct a perfect subset.
Strict returns persist under the nonlinear perturbation of the circle
and yield a branching construction of positive Hausdorff dimension.
An exact weighted-return certificate then gives the quantitative lower
bound in Theorem~\ref{thm:circle-quantitative}.

\begin{definition}\label{def:cantor-circle}
For the following unit-circle construction, write
\[
C:=C_{1/3}
=\left\{\sum_{j=1}^{\infty}\frac{2\varepsilon_j}{3^j}:
                \varepsilon_j\in\{0,1\}\right\},
\qquad
E:=E_{2,1/3}.
\]
\end{definition}

The proof uses a specific Cantor cylinder. Consider points of the form
\[
x=\frac{566+b}{3^7},\qquad y=\frac{6336+a}{3^8},
\qquad a,b\in C.
\]
The integers \(566\) and \(6336\), written as ternary words of lengths
seven and eight, are \(0202222\) and \(22200200\).
The circle equation in these coordinates is
\[
(6336+a)^2+9(566+b)^2=3^{16}.
\]
At \((a,b)=(0,0)\) and \((1,1)\), the left side minus \(3^{16}\)
equals \(-18621\) and \(4249\), respectively.
These signs alone only show a crossing of the ordinary square.
Lemma~\ref{lem:interval-cover} proves that every retained Cantor square maps
onto its entire corner-value interval, which makes the crossing a
point of the Cantor product.

The second ingredient concerns isolated points. Throughout the selected
cylinder, the normal ratio \(y/(3x)\) lies strictly between \(11/9\)
and \(2027/1624\). In this interval, a four-digit expansion that is
asymptotically confined to the regions with only one possible first
digit must eventually be periodic. We prove this assertion directly
in Lemma~\ref{lem:periodicity}. An exact quadratic expansion then transfers
it to any hypothetical isolated circle point.
Such a point would have rational coordinates, whereas
Lemma~\ref{lem:rational} proves that the only rational points of \(E\)
are the two axis points. Thus the selected compact intersection has
no isolated points.

This argument establishes cardinality. It does not give a positive
Hausdorff dimension bound, nor does it exclude quadratic irrational
numbers from \(C\). Both distinctions matter: a perfect compact set
can have zero Hausdorff dimension, and periodicity here follows from
isolation in a particular two-coordinate level set.

\paragraph{Rational points on the unit circle.}\label{sec:rational}

We first record the arithmetic obstruction that will rule out isolated
points at the end of the proof.

\begin{lemma}\label{lem:rational}
The rational points of \(E\) are exactly \((0,1)\) and \((1,0)\).
\end{lemma}

\begin{proof}
Suppose that \(x,y\) are rational, belong to \(C\), and satisfy
\(x^2+y^2=1\). If either coordinate is zero, the conclusion follows.
We may therefore assume that \(0<x,y<1\).
Choose a common denominator and write
\[
x=\frac{p}{d},\qquad y=\frac{q}{d},\qquad
p^2+q^2=d^2,\qquad \gcd(p,q,d)=1.
\]
If \(3\mid d\), reduction modulo \(3\) forces \(3\mid p,q\), a
contradiction. Hence \(d\) is coprime to \(3\).

Both ternary expansions are consequently purely periodic.
Indeed, choose \(k\geq1\) such that \(3^k\equiv1\pmod d\).
Multiplication by \(3^k\) preserves the fractional part of each
coordinate. Neither coordinate is a noninteger triadic rational, so
there is no endpoint ambiguity in its expansion.
Since the digits belong to \(\{0,2\}\), dividing the two period
numerators by \(2\) gives
\[
x=\frac{A}{N},\qquad y=\frac{B}{N},
\qquad N:=\frac{3^k-1}{2},
\]
where the ternary digits of \(A,B\) belong to \(\{0,1\}\).
All \(k\) ternary digits of \(N\) are \(1\), and
\[
A^2+B^2=N^2.
\]
Modulo \(3\), exactly one of \(A,B\) is \(1\).
After interchanging the coordinates, assume that
\(A\equiv1\pmod3\) and \(B\equiv0\pmod3\).

Let \(j\geq1\) be the exact exponent of \(3\) dividing \(B\).
The identity
\[
(N-A)(N+A)=B^2
\]
and \(N+A\equiv2\pmod3\) show that the exact exponent of \(3\)
dividing \(N-A\) is \(2j\).
Subtracting \(A\) from the all-one ternary expansion of \(N\)
requires no borrowing. Thus \(N-A\) has only \(0\) and \(1\) as
ternary digits, and its first nonzero digit is \(1\).
It follows that
\[
\frac{N-A}{3^{2j}}\equiv1\pmod3.
\]
Dividing the product identity by \(3^{2j}\) now gives a contradiction
modulo \(3\): its left side is \(1\cdot2\), while its right side is
a nonzero square. This excludes \(0<x,y<1\).
\end{proof}

\paragraph{A nonempty local intersection.}\label{sec:source}

We need a crossing criterion for actual Cantor points, valid in every
descendant square. The following elementary covering argument provides
it.

Let \(G\) be continuously differentiable on \([0,1]^2\). For ternary
words of the same length \(n\), with digits in \(\{0,2\}\), let \(u,v\)
be the corresponding left endpoints. The associated square and Cantor
product are
\[
Q:=[u,u+3^{-n}]\times[v,v+3^{-n}],
\qquad
(u+3^{-n}C)\times(v+3^{-n}C).
\]

\begin{lemma}\label{lem:interval-cover}
Suppose that
\[
0<\min_{[0,1]^2}G_b
\leq\max_{[0,1]^2}G_b
<\min_{[0,1]^2}G_a,
\qquad
\max_{[0,1]^2}G_a<3\min_{[0,1]^2}G_b.
\]
Then every square \(Q\) described above satisfies
\[
G\bigl((u+3^{-n}C)\times(v+3^{-n}C)\bigr)
=[G(u,v),G(u+3^{-n},v+3^{-n})].
\]
\end{lemma}

\begin{proof}
Write \(h\) for the side length of a retained square with lower-left
corner \((u,v)\). Its four children have lower-left offsets
\[
(0,0),\quad(0,2h/3),\quad(2h/3,0),\quad(2h/3,2h/3).
\]
Positivity of both partial derivatives implies that the image of each
ordinary child square is the interval between its lower-left and
upper-right corner values. Both sets of endpoints occur in the
displayed order, since increasing the first coordinate by \(2h/3\)
and decreasing the second by \(2h/3\) increases \(G\).

The first two intervals overlap because
\[
G(u+h/3,v+h/3)-G(u,v+2h/3)
\geq\frac h3\bigl(\min G_a-\max G_b\bigr)>0.
\]
The middle two overlap because
\[
G(u+h/3,v+h)-G(u+2h/3,v)
\geq h\min G_b-\frac h3\max G_a>0.
\]
The last two overlap by the same estimate as the first two.
Their union is therefore the whole parent corner-value interval.

Given any value in this interval, choose a child whose corner-value
interval contains it and repeat. The resulting nested closed squares
have diameters tending to zero. Their common point belongs to the
retained Cantor product, and continuity gives the prescribed value.
The same argument applies inside every descendant square.
\end{proof}

For the remainder of this construction, define
\[
F(a,b):=(6336+a)^2+9(566+b)^2
\]
and
\[
S:=\left\{\left(\frac{566+b}{3^7},\frac{6336+a}{3^8}\right):
                       a,b\in C,\ F(a,b)=3^{16}\right\}.
\]
The two allowed ternary prefixes show that \(S\subset E\).

\begin{proposition}\label{prop:source}
The set \(S\) is nonempty and compact.
\end{proposition}

\begin{proof}
On \([0,1]^2\),
\[
12672\leq F_a\leq12674,\qquad
10188\leq F_b\leq10206.
\]
In particular, \(12672>10206\) and \(12674<3\cdot10188\), so
Lemma~\ref{lem:interval-cover} applies. Since
\[
F(0,0)-3^{16}=-18621,\qquad
F(1,1)-3^{16}=4249,
\]
there is a point of \(C\times C\) with value \(3^{16}\).
Compactness follows from continuity and compactness of \(C\times C\).
\end{proof}

We will use two consequences of this calculation.
First, a retained square contains a point of
\(F(a,b)=3^{16}\) in its Cantor product if and only if its corner
values bracket \(3^{16}\), with equality allowed.
Second, throughout the entire source rectangle,
\begin{equation}\label{eq:normal-band}
\frac{11}{9}<\frac{704}{567}
\leq\frac{6336+a}{9(566+b)}
=\frac{y}{3x}
\leq\frac{6337}{5094}<\frac{2027}{1624}.
\end{equation}
The strict endpoint margins are \(11/567\) and
\(17125/4136328\), respectively.

\paragraph{Eventual periodicity of digit sequences.}\label{sec:digits}

Fix a real number
\begin{equation}\label{eq:r-band}
\frac{11}{9}<r<\frac{2027}{1624}.
\end{equation}
For a sequence \(c_n\in\{0,1,r,r+1\}\), indexed by \(n\geq0\), define
\[
T_n:=\sum_{j=0}^{\infty}\frac{c_{n+j}}{3^{j+1}}.
\]
Thus \(T_{n+1}=3T_n-c_n\).
The intervals relevant to a single admissible first digit are
\begin{equation}\label{eq:exclusive-domains}
\begin{aligned}
I_0&:=[0,1/3],
&I_1&:=[(r+1)/6,r/3],\\
I_r&:=[(r+3)/6,(r+1)/3],
&I_{r+1}&:=[(3r+1)/6,(r+1)/2].
\end{aligned}
\end{equation}
They are closed: the argument below must allow errors from curvature
and cannot discard their endpoints.

\begin{lemma}\label{lem:periodicity}
For \(r\) satisfying \eqref{eq:r-band}, the sequence \((c_n)\) is
eventually periodic whenever
\[
\lim_{n\to\infty}\operatorname{dist}(T_n,I_{c_n})=0.
\]
\end{lemma}

\begin{proof}
Put \(\delta_n:=\operatorname{dist}(T_n,I_{c_n})\).
Every exclusion below has a strictly positive margin depending only
on the fixed parameter \(r\). Because
\(\displaystyle\lim_{n\to\infty}\delta_n=0\), these exclusions hold
simultaneously after a sufficiently large index.
The reflection
\[
T_n\longmapsto\frac{r+1}{2}-T_n,\qquad
c_n\longmapsto r+1-c_n
\]
preserves the hypotheses and exchanges the digits \(0,r+1\) and \(1,r\).

If \(c_n=1\), then \(T_{n+1}\leq r-1+3\delta_n\).
Every domain for a nonzero next digit has lower endpoint at least
\((r+1)/6\). Hence \(c_{n+1}\ne0\) would imply
\[
\frac{7-5r}{6}\leq3\delta_n+\delta_{n+1}.
\]
The left side is positive, so sufficiently late occurrences of \(1\)
must be followed by \(0\). Reflection shows that sufficiently late
occurrences of \(r\) must be followed by \(r+1\).

After the pair \(10\), the recurrence gives
\[
\frac{3(r-1)}2-9\delta_n
\leq T_{n+2}=9T_n-3
\leq3(r-1)+9\delta_n.
\]
This excludes the next digit \(0\) or \(r+1\) at sufficiently late
indices, since the respective separating margins are
\[
\frac{3(r-1)}2-\frac13=\frac{9r-11}{6}>0,
\qquad
\frac{3r+1}{6}-3(r-1)=\frac{19-15r}{6}>0.
\]
The latter inequality uses \(2027/1624<19/15\).
The reflected conclusion holds after the pair \(r(r+1)\).
Consequently, once a middle digit \(1\) or \(r\) occurs sufficiently
late, the entire subsequent sequence consists of the two blocks
\[
10,\qquad r(r+1).
\]

Call these the low block and the high block.
An infinite sequence of low blocks has value \(3/8\).
If three successive blocks are low, high, high, then replacing all
other blocks by low blocks gives the lower bound
\[
T_n\geq\frac38+(4r-2)\left(\frac1{81}+\frac1{729}\right)
       =\frac38+\frac{20(2r-1)}{729}.
\]
Its first digit is \(1\), which also requires \(T_n\leq r/3+\delta_n\).
But
\[
\frac38+\frac{20(2r-1)}{729}-\frac r3
=\frac{2027-1624r}{5832}>0.
\]
Thus low, high, high is eventually forbidden, and reflection also
forbids high, low, low.
A binary sequence avoiding these two words is eventually constant
or alternating: after its first change, either forbidden word would
appear if the next symbol did not change again.
The original digit sequence is therefore eventually periodic.

It remains to consider a tail containing only \(0\) and \(r+1\).
At an occurrence of \(0,(r+1),(r+1)\), the terminal digit's domain
and the recurrence give
\[
T_n=\frac{r+1+T_{n+2}}9
\geq\frac{9r+7}{54}-\frac{\delta_{n+2}}9.
\]
The initial digit \(0\) requires \(T_n\leq1/3+\delta_n\).
These inequalities contradict one another for sufficiently large \(n\),
because
\[
\frac{9r+7}{54}-\frac13=\frac{9r-11}{54}>0.
\]
Reflection excludes \((r+1),0,0\) as well.
The same binary-word argument proves eventual constancy or alternation
in this case.
\end{proof}

The hypothesis of Lemma~\ref{lem:periodicity} permits arbitrary
vanishing errors. In particular, it applies when the domains arise
as linear approximations to a curved level set. This is the reason
for retaining the error terms throughout the proof.

\paragraph{The perfect subset and uncountability.}\label{sec:main}

We now combine the local covering property with the digit lemma.

\begin{theorem}\label{thm:main}
The set \(E\) has cardinality \(2^{\aleph_0}\).
\end{theorem}

\begin{proof}
By Proposition~\ref{prop:source}, \(S\) is a nonempty compact subset
of \(E\). We will show that \(S\) has no isolated points.
Suppose to the contrary that
\[
(x,y)=\left(\frac{566+b}{3^7},\frac{6336+a}{3^8}\right)\in S
\]
is isolated. Fix its normal ratio
\[
r:=\frac{6336+a}{9(566+b)}=\frac{y}{3x}.
\]
By \eqref{eq:normal-band}, this fixed value satisfies
\eqref{eq:r-band}.

Let \(a_n,b_n\in C\) be the tails of the allowed ternary expansions
of \(a,b\) after their first \(n\) digits. Write their next digits as
\(2\xi_n,2\beta_n\), where \(\xi_n,\beta_n\in\{0,1\}\).
Then
\[
a_n=\frac{2\xi_n+a_{n+1}}3,\qquad
b_n=\frac{2\beta_n+b_{n+1}}3.
\]
Define
\[
T_n:=\frac{ra_n+b_n}{2},\qquad c_n:=r\xi_n+\beta_n.
\]
These are precisely the tails and digits introduced before Lemma~\ref{lem:periodicity};
in particular, \(T_{n+1}=3T_n-c_n\) holds exactly.

The retained square of depth \(n\) containing \((a,b)\) can be
parametrised by
\[
\bigl(a+3^{-n}(s-a_n),\,b+3^{-n}(t-b_n)\bigr),
\qquad s,t\in[0,1].
\]
The quadratic expansion is exact:
\begin{align}
&\frac{F\bigl(a+3^{-n}(s-a_n),\,b+3^{-n}(t-b_n)\bigr)-F(a,b)}
       {36(566+b)3^{-n}}\notag\\
&\hspace{12mm}=\frac{rs+t}{2}-T_n
 +\frac{3^{-n}\bigl((s-a_n)^2+9(t-b_n)^2\bigr)}{36(566+b)}.
\label{eq:taylor}
\end{align}
The final term lies between zero and
\[
\eta_n:=\frac{5}{10188}\,3^{-n}.
\]
The normal \(r\) is held fixed at the original point throughout this
calculation; it is not replaced by a normal at a cylinder corner.

For a child with digit label \(c\in\{0,1,r,r+1\}\), the linear function
\((rs+t)/2\) takes values in
\[
J_c:=\left[\frac c3,\frac c3+\frac{r+1}{6}\right].
\]
If \(T_n\) is more than \(\eta_n\) from both endpoints inside \(J_c\),
the left side of \eqref{eq:taylor} has opposite signs at the child's
lower-left and upper-right corners. Lemma~\ref{lem:interval-cover}
then supplies an actual point of the Cantor product on the level
\(F=3^{16}\) in this child.

Because the original point is isolated, every sufficiently deep
retained square containing it has exactly one child containing a
point of that level set. The retained intervals at any fixed depth
are disjoint, including their endpoints, so a point in a different
child would be a different point of \(S\).
The original point belongs to its prescribed child, and therefore
\(T_n\in J_{c_n}\) exactly.

We claim that for all sufficiently large \(n\),
\begin{equation}\label{eq:approx-exclusive}
\operatorname{dist}(T_n,I_{c_n})\leq\eta_n.
\end{equation}
Here the \(I_c\) are the closed intervals in
\eqref{eq:exclusive-domains}.
Choose \(n\) so large that
\(\eta_n<\min\{1/3,(r-1)/3\}\).
If \(c_n=0\) and \(T_n>1/3+\eta_n\), then \(T_n\) is more than
\(\eta_n\) inside \(J_1\); its distance from the other endpoint of
\(J_1\) is at least \(1/3\).
This would give a point in a competing child.
If \(c_n=1\), the alternatives
\[
T_n<(r+1)/6-\eta_n
\quad\hbox{or}\quad
T_n>r/3+\eta_n
\]
place \(T_n\) more than \(\eta_n\) inside \(J_0\) or \(J_r\),
respectively. The margins at their other endpoints are at least
\(1/3\) and \((r-1)/3\).
Both alternatives again give a competing child.
The cases \(c_n=r\) and \(c_n=r+1\) follow by reflecting the linear
intervals. This last argument only uses the uniform absolute error
bound in \eqref{eq:taylor}; it does not assume reflection symmetry
of the quadratic term.
This proves \eqref{eq:approx-exclusive}, with no omitted endpoint
case.

Since \(\displaystyle\lim_{n\to\infty}\eta_n=0\),
Lemma~\ref{lem:periodicity} makes \((c_n)\) eventually periodic.
The four values \(0,1,r,r+1\) are distinct, so each digit \(c_n\)
uniquely determines \((\xi_n,\beta_n)\).
Thus the ternary expansions of both \(a\) and \(b\) are eventually
periodic, and \(x,y\) are rational.
Lemma~\ref{lem:rational} excludes such a point in \(S\), whose two
coordinates are strictly between zero and one. This contradiction
shows that \(S\) is perfect.

For completeness, a nonempty perfect compact subset of the plane
contains continuum many points. Inside a relative neighbourhood
meeting the set, choose two distinct points and smaller disjoint
closed neighbourhoods around them. Each can be chosen to have
nonempty relative interior. Repeat in each neighbourhood, with
diameters decreasing to zero. Compactness gives a limiting point
for every infinite binary sequence, and disjointness makes those
points distinct. Hence \(\#S\geq2^{\aleph_0}\).
Since \(S\subset E\subset\mathbb R^2\), the reverse cardinality bound
is automatic, proving the theorem.
\end{proof}

The construction does not require immediate branching at every scale.
The first subdivision gives a concrete illustration.

\begin{example}\label{ex:first-subdivision}
For a child with first ternary digits \(d,e\in\{0,2\}\) in the
\(a,b\) coordinates, write
\(a=(d+s)/3\), \(b=(e+t)/3\), with \(s,t\in[0,1]\).
Multiplying \(F(a,b)-3^{16}\) by \(9\) gives
\[
(19008+d+s)^2+9(1698+e+t)^2-3^{18}.
\]
Its corner-value intervals are
\[
\begin{array}{c@{\qquad}r@{\ ,\ }l}
\toprule
(d,e)&\multicolumn{2}{c}{\text{corner-value interval}}\\
\midrule
(0,0)&[-167589&-98999]\\
(0,2)&[-106425&-37799]\\
(2,0)&[-91553&-22959]\\
(2,2)&[-30389&38241]\\
\bottomrule
\end{array}
\]
Only the child \((2,2)\) meets the level zero.
Lemma~\ref{lem:interval-cover} makes this a statement about actual
Cantor points, not merely ordinary-square crossings.
The theorem proves that the surviving compact intersection is
perfect despite this first subdivision having a single surviving
child.
\end{example}

\paragraph{Strict returns and positive dimension.}

\begin{lemma}\label{lem:strict-returns}
Let \(L(s,t)\) be a positive linear function and \(B\) a compact
set of levels such that every \(u\in B\) has two incomparable
nonempty Cantor words returning its normalized level into
\(\operatorname{int}B\).
Then every sufficiently small Lipschitz perturbation of \(L\),
at every level sufficiently close to \(B\), has a positive-dimensional
zero set in \(C_\lambda^2\).
\end{lemma}
\begin{proof}
For a word \(w\), let \(z_w\) be its lower corner and let
\(\Phi_w(u)=\lambda^{-|w|}(u-L(z_w))\).
The two strict inequalities for a chosen pair persist on an
open interval about \(u\). Compactness gives a finite collection
of pairs, an open neighborhood \(U\supset B\), a maximum length
\(M\), and a margin \(\eta>0\) such that every \(u\in U\) has
two selected returns at distance at least \(\eta\) from
\(\mathbb R\setminus B\).

Consider \(L+e-u=0\), where \(e(0,0)=0\) and
\(\operatorname{Lip}(e)<\eta\lambda^M/2\) in the sup norm.
On restricting to \(w\), the new error is
\[
\lambda^{-|w|}
\bigl(e(z_w+\lambda^{|w|}(s,t))-e(z_w)\bigr).
\]
Its Lipschitz constant does not increase. The new level differs
from \(\Phi_w(u)\) by \(\lambda^{-|w|}e(z_w)\), of absolute
value less than \(\eta/2\). Both words therefore return to
\(B\subset U\), and the process repeats.

Every infinite branch converges to a zero: the original residual
is the cylinder scale times a uniformly bounded normalized
residual. Assign mass \(2^{-j}\) to each binary generation-\(j\)
cylinder. Its depth is at most \(jM\), and distinct such cylinders
are separated by at least
\((1-2\lambda)\lambda^{jM-1}\).
A ball of radius comparable to \(\lambda^{jM}\) meets a bounded
number of them. Choosing \(j\) with
\(\lambda^{(j+1)M}<r\leq\lambda^{jM}\) gives
\[
\nu(B(z,r))\leq C r^{\log2/(M\log(1/\lambda))}.
\]
The mass distribution principle~\cite{Mattila1995} completes the proof.
\end{proof}

\begin{lemma}\label{lem:good-slopes}
The eventual-periodicity conclusion of Lemma~\ref{lem:periodicity},
including vanishing errors, holds when the slope belongs to any of
\[
\left(\frac{89}{73},\frac{171}{137}\right),\
\left(\frac{101}{79},\frac{183}{143}\right),\
\left(\frac{89}{57},\frac{171}{89}\right),\
\left(\frac{143}{61},\frac{237}{101}\right),\
\left(\frac{137}{57},\frac{219}{89}\right).
\]
\end{lemma}
\begin{proof}
We give the finite criterion and its exact verification in
Appendix~\ref{app:certificates}. Label \(0,1,r,r+1\) by \(0,1,2,3\).
For each of the \(4^5\) words of length five, impose the five
closed exclusive-domain conditions from
\eqref{eq:exclusive-domains}, after pulling them back to the
initial level. Every endpoint is affine in \(r\), so simultaneous
feasibility is a finite list of rational linear inequalities in
\(r\). Keep a word if this feasible slope interval meets the
specified open band.

Make a directed graph whose vertices are words of length four
and whose edges are the retained words of length five. For the
five bands in the stated order, the numbers of edges are
\(96,80,160,92,112\); the numbers of recurrent cycles are
\(7,6,8,6,7\). The exact certificate assigns an integer rank to
each vertex, nondecreasing along every edge; within each rank
the recurrent edges form disjoint directed simple cycles.
An infinite path therefore eventually stays on a single cycle.

Fix a slope in the open band. Each discarded word has a strictly
positive feasibility gap at this slope. There are only finitely
many words, so the minimum of these gaps is positive.
Errors tending to zero cannot produce a discarded word at
arbitrarily late positions. The resulting address is eventually
a graph path, and hence eventually periodic.
\end{proof}

Call the interval in \eqref{eq:r-band} and the five intervals of
Lemma~\ref{lem:good-slopes} the admissible slope intervals.
Reciprocal slopes are treated by interchanging the coordinates.

\begin{lemma}\label{lem:irrational-tangent}
Suppose a \(C^1\) level curve in a Cantor rectangle contains a
nonempty compact perfect subset of \(C^2\).
If at one of its points the positive normalized derivative ratio
is irrational and lies in an admissible slope interval, then
the Cantor intersection of the curve has positive Hausdorff dimension.
\end{lemma}
\begin{proof}
Fix that ratio \(r\), and write \(L_r(s,t)=(rs+t)/2\).
Its four child image intervals are
\[
\left[\frac c3,\frac c3+\frac{r+1}{6}\right],
\qquad c\in\{0,1,r,r+1\}.
\]
For \(1<r<3\) they cover \([0,(r+1)/2]\); iterating this
coverage proves \(L_r(C^2)=[0,(r+1)/2]\).
An isolated point of any affine fibre eventually has a unique
continuing child. Its tails consequently lie in the closed
exclusive domains. Lemmas~\ref{lem:periodicity} and
\ref{lem:good-slopes} make its coordinates rational.

A countable nonempty compact metric space has an isolated point:
otherwise repeated choices of two disjoint neighborhoods give
a binary nested construction. If an affine fibre were countable
with two distinct points, two disjoint Cantor cylinders would
therefore each contain an isolated point of that fibre.
Their coordinates would be rational, and equality of levels
would make \(r\) rational. Thus every nonempty affine fibre is
either a singleton or uncountable. Every level in the compact
switch set
\[
B_r=
\left[\frac13,\frac{r+1}{6}\right]\cup
\left[\frac r3,\frac{r+3}{6}\right]\cup
\left[\frac{r+1}{3},\frac{3r+1}{6}\right]
\]
has at least two points, hence uncountably many.

In an uncountable fibre, some descendant has two uncountable
children. Otherwise the fibre is the union of one limiting point
and countably many countable offshoots. A node with two
uncountable children has normalized level in the interior of
\(B_r\): at a switch endpoint one child is an extreme singleton.
Apply this argument again within each of the two children.
It produces two incomparable strict return words from every
level of \(B_r\). Compactness supplies finitely many word pairs,
a maximum length, and a common positive return margin.
Lemma~\ref{lem:strict-returns} makes these returns robust under
a sufficiently small signed Lipschitz perturbation of \(L_r\).

It remains to find a suitable normalized cylinder on the given
curve. In the depth-\(j\) square containing the specified perfect
set point, divide the defining function by
\(2\,3^{-j}F_t\) evaluated at that point. The normalized equation
is \(L_r+e_j-u_j=0\), where \(e_j(0,0)=0\) and
\(\operatorname{Lip}(e_j)\to0\), by continuity of the gradient.
There are arbitrarily deep branching squares on its address;
otherwise that point would be isolated in the perfect set.
At these depths two child image intervals come within
\(\|e_j\|_\infty\) of \(u_j\). Nonadjacent intervals are separated
by \((r-1)/6\), so eventually
\(\operatorname{dist}(u_j,B_r)\leq\|e_j\|_\infty\).
A sufficiently deep branching square lies in the neighborhood
where the robust returns apply. Their binary construction has
positive Hausdorff dimension.
\end{proof}

In particular, the perfect subset in Theorem~\ref{thm:main}
has positive dimension. Indeed, the normalized tangent ratio
varies strictly monotonically along the positive unit-circle arc,
so only countably many of its points have rational ratio.
The perfect set is uncountable, and all its ratios lie in
\eqref{eq:r-band}.

\paragraph{A quantitative lower bound for the unit circle.}

\begin{theorem}\label{thm:circle-quantitative}
For the unit circle in the classical Cantor dust,
\[
\frac{\log(57/50)}{\log3}\leq\dim_{\mathrm H}(E_{2,1/3})
                  \leq\frac{2\log2}{\log3}-1 .
\]
\end{theorem}
\begin{proof}
The upper bound is Theorem~\ref{thm:dimension-upper}.
For the lower bound use the Cantor cylinder
\[
x=\frac{P+b}{3^{59}},\qquad y=\frac{Q+a}{3^{60}},
\]
where
\[
\begin{aligned}
P&=3663407072939048637165799292,\\
Q&=40941731002436775062259738416 .
\end{aligned}
\]
The padded ternary expansions of \(P,Q\) contain only \(0,2\).
The normalized circle equation is
\[
\frac{ra+b}{2}+e(a,b)-T=0,\qquad
r=\frac Q{9P},\quad
e(a,b)=\frac{a^2+9b^2}{36P},\quad
T=\frac{3^{120}-Q^2-9P^2}{36P}.
\]
Here \(e(0,0)=0\) and
\(\operatorname{Lip}(e)\leq5/(9P)\).
Subtracting the lower-corner remainder and dividing by the
side length in any descendant preserves this bound.

The exact weighted-return certificate of
Appendix~\ref{app:certificates} covers all three components of
\(B_r\), assigns them weights \(1,2,1\), and uses return words
of length at most \(33\). The initial \(T\) lies in the middle
component. For every current level in component \(i\), its
incomparable return words \(w\), of lengths \(|w|\) and
destination components \(i(w)\), satisfy
\[
\sum_w(57/50)^{-|w|}\,\operatorname{weight}(i(w))
 \ \geq\frac{20001}{20000}\operatorname{weight}(i).
\]
All endpoints and all error enclosures in this assertion are
rational. In particular, the grid calculation includes
\(3\operatorname{Lip}(e)\) at every digit; it therefore verifies
the actual nonlinear returns, uniformly at every descendant.

Assign return probability proportional to
\((57/50)^{-|w|}\operatorname{weight}(i(w))\).
The preceding inequality shows that along a sequence of returns
of total ordinary depth \(d\), its mass is at most
\((57/50)^{-d}\) times the ratio of terminal and initial weights.
This ratio is at most two. At an intermediate ordinary depth,
stopping just before the next return loses at most \(33\) digits,
so every ordinary depth-\(j\) cylinder has mass at most
\(2(57/50)^{33}(57/50)^{-j}\).
The antichain property makes this assertion valid for all branches
simultaneously. A ball of radius comparable to \(3^{-j}\) meets
a bounded number of Cantor squares. The resulting measure is
therefore Frostman with exponent \(\log(57/50)/\log3\).
The normalized levels and errors remain bounded, so their nested
limits satisfy the exact circle equation, as in
Lemma~\ref{lem:strict-returns}.
\end{proof}

%%%%%%%%%%%%%%%%%%%%%%%%%%%%%%%%%%%%%%%%%%%%%%%5
\section{Varying the exponent \texorpdfstring{\(\alpha\)}{alpha}}\label{sec:first}

We first prove positive dimension at the middle-third parameter for
all real exponents, and then study how the first transition depends
on the arithmetic of the exponent. The principal conclusions are
\[
\begin{cases}
\dim_{\mathrm H}(E_{\alpha,1/3})>0,
  & \alpha\in\mathbb R,\ \alpha>1,\\[4pt]
\lambda_{\min}(\alpha)=1/4,
  & \alpha\notin\mathbb Q,\ \alpha>1,\\[4pt]
0\leq\lambda_{\min}(\alpha)-1/4
  \leq\dfrac{(3/2)^\alpha\log4}{2\alpha b},
  & \alpha=a/b\in\mathbb Q\setminus\mathbb Z,\ \alpha\geq2,
\end{cases}
\]
where \(a/b\) is in lowest terms. The first assertion concerns the fixed
middle-third Cantor set, whereas the other two assertions describe the first
Hausdorff phase transition as the contraction ratio varies. The rational
estimate also holds for integers; it does not by itself imply a strict
gap above \(1/4\).

\subsection{All real exponents at the middle-third parameter}\label{subsec:classical}

\begin{theorem}\label{thm:integer-powers}\label{thm:real-powers}
For every real \(\alpha>1\),
\[
\dim_{\mathrm H}\{(x,y)\in C\times C:x^\alpha+y^\alpha=1\}>0.
\]
\end{theorem}
\begin{proof}
The proof for irrational exponents is analytic. For rational exponents,
we combine a logarithmic source criterion with exact finite covers and
strict-return certificates. Appendix~\ref{app:real-certificates} gives
the verification predicates and the accompanying certificate files.

\emph{Irrational exponents.}
For \(1/3\leq q\leq3\), the four first-level intervals of
\(C-qC\) cover \([-q,1]\). Iteration gives \(C-qC=[-q,1]\), and hence
\[
\frac{(2+C)/3}{(2+C)/3}=[2/3,3/2].
\]
Indeed, \((2+s)/(2+t)=q\) is equivalent to
\(s-qt=2(q-1)\). Quotients of positive finite ternary left endpoints
in \((2+C)/3\) are therefore dense in this interval.

Choose an integer \(h\) such that \(\alpha3^h\in[1,3)\).
The two open intervals
\[
\left(\frac{41}{50},\frac{369}{200}\right),\qquad
\left(\frac{200}{123},\frac{150}{41}\right)
\]
cover \([1,3]\). The first permits multiplication by a quotient
in \((2/3,3/2)\) to reach \(123/100\); the second permits reaching
\(100/123\) after replacing \(h\) by \(h-1\).
By density, we may choose finite ternary left endpoints
\(0<u_0,v_0<1\) for which
\[
r:=\alpha\frac{v_0}{u_0}3^h
\]
belongs to \((11/9,2027/1624)\) or its reciprocal.
It is irrational because \(\alpha\) is irrational.
The affine part of the proof of Lemma~\ref{lem:irrational-tangent},
with coordinates interchanged if necessary, supplies a compact set
of levels of \(rs+t\), each having two incomparable strict returns.
Fix one such level \(\theta\).

Take large integers \(i,j\) with \(i-j=h\), beyond the terminating
lengths of \(v_0,u_0\), respectively, and set
\[
u=u_0+3^{-j}s,\qquad v=v_0+3^{-i}(1-t),\qquad
A:=\frac{u_0^\alpha}{v_0+3^{-i}(1-\theta)}.
\]
These are genuine Cantor cylinders, and \(A>0\) for large \(i\).
The exact identity
\[
\frac{u^\alpha-Av}{A3^{-i}}=rs+t-\theta+e(s),\qquad e(0)=0,
\]
has
\[
e'(s)=\frac{\alpha3^{i-j}(u_0+3^{-j}s)^{\alpha-1}}A-r.
\]
As \(i,j\to\infty\) with \(i-j=h\), this derivative tends uniformly
to zero. Fix a sufficiently deep pair of cylinders so that
Lemma~\ref{lem:strict-returns} applies, with room for a further
small \(C^1\) perturbation.

Irrational rotation provides positive integers \(k,m\to\infty\)
such that \(m-\alpha k\to\log_3(\alpha/A)\). Taylor's theorem gives
\[
3^{\alpha k}\bigl[(3^{-k}u)^\alpha+(1-3^{-m}v)^\alpha-1\bigr]
\longrightarrow u^\alpha-Av
\quad\text{in }C^1
\]
on the fixed rectangle. After division by \(A3^{-i}\), strict
returns therefore persist for sufficiently large \(k,m\).
Both \(3^{-k}u\) and \(1-3^{-m}v\) belong to \(C\), and the
coordinate map is bi-Lipschitz. This proves the irrational case.

\emph{Rational exponents and strict logarithmic sources.}
Write \(\alpha=a/b>1\) in lowest terms. The case \(\alpha=2\) follows from
Theorem~\ref{thm:circle-quantitative}. For every other such exponent,
there are no solutions with both coordinates positive and rational.
To see this, suppose otherwise. In a normal closure, every conjugate
of \(x^{a/b}\) is \(x^{a/b}\) times a \(b\)-th root of unity, and
likewise for \(y^{a/b}\). The conjugate equation and equality in the
triangle inequality force both conjugates to remain positive real.
Thus \(x^{a/b},y^{a/b}\) are rational. Prime valuations and
\(\gcd(a,b)=1\) imply \(x=u^b,y=v^b\) for positive rationals
\(u,v\). Since \(a\geq3\), the equation \(u^a+v^a=1\) contradicts
Fermat's Last Theorem~\cite{Wiles1995,TaylorWiles1995}.

Let \(u=(P+s)/3^p\), \(v=(Q+t)/3^q\), where \(P,Q>0\) have
padded ternary digits in \(\{0,2\}\). A strict logarithmic source
consists of \(d\in b^{-1}\mathbb Z\) satisfying
\begin{align*}
d\log3+\alpha\log(P/3^p)-\log \alpha-\log((Q+1)/3^q)&<0,\\
d\log3+\alpha\log((P+1)/3^p)-\log \alpha-\log(Q/3^q)&>0,
\end{align*}
with
\[
\left[\frac{\alpha Q}{P+1},\frac{\alpha(Q+1)}P\right]
\]
strictly inside an admissible slope interval or its reciprocal.
If \(d=j/b\), choose arbitrarily large integers \(k\) satisfying
\(ak+j\equiv0\pmod b\). Then \(\alpha k+d\) is an integer and
\[
x=3^{-k}u,\qquad y=1-3^{-\alpha k-d}v
\]
are Cantor cylinders. The exact logarithmic residual
\[
d\log3+\alpha\log u-
\log\!\left(3^{\alpha k+d}[1-(1-3^{-\alpha k-d}v)^\alpha]\right)
\]
converges in \(C^1\) to \(d\log3+\alpha\log u-\log \alpha-\log v\).
For real powers the required identity is
\[
\frac{1-(1-z)^\alpha}{z}=\alpha\int_0^1(1-tz)^{\alpha-1}\,dt.
\]
Thus the strict corner and slope inequalities persist. Reflecting
the \(v\) coordinate makes both partial derivatives positive;
Lemma~\ref{lem:interval-cover} gives a nonempty compact fibre.
If it had an isolated point, the vanishing-error digital argument
of Lemmas~\ref{lem:periodicity} and~\ref{lem:good-slopes}, as in
Theorem~\ref{thm:main}, would make both coordinates rational.
The preceding obstruction excludes this. The fibre is perfect,
and its tangent ratio is strictly monotone along the Fermat arc.
Lemma~\ref{lem:irrational-tangent} gives positive dimension.

\emph{Coverage of the rational exponents.}
The interval of permitted \(d\) has length
\[
\frac{\alpha\log(1+1/P)+\log(1+1/Q)}{\log3}.
\]
For denominators \(b\geq110\), a uniform construction makes this
length greater than \(1/110\). Write \(\alpha=3^m t\), \(1\leq t<3\).
The sixteen rows \((P,Q,\ell,u)\) in
Appendix~\ref{app:real-certificates} satisfy the following exact
open-cover condition:
\[
[1,3]\subset\bigcup
\left(\max\left\{\frac{\ell(P+1)}Q,
\left(\frac{11}{1100}-\frac1{Q+1}\right)(P+1)\right\},
\frac{uP}{Q+1}\right).
\]
Replace \(P\) by \(3^mP\), appending zero ternary digits.
The derivative interval stays strictly in \((\ell,u)\), and
\[
\alpha\log\left(1+\frac1{3^mP}\right)
+\log\left(1+\frac1Q\right)
>\frac{t}{P+1}+\frac1{Q+1}>\frac{11}{1100}.
\]
Since \(\log3<11/10\), the permitted interval contains a point of
\(b^{-1}\mathbb Z\) whenever \(b\geq110\).

Integer exponents are covered by the certificates and nine direct
returns detailed in Appendix~\ref{app:integer-certificates}.
For \(2\leq b\leq109\), exact verification gives the following
exhaustive decomposition:
\[
\begin{array}{c|l}
\text{exponent range}&\text{coverage}\\ \hline
1<a/b<2187&7{,}738{,}063\text{ width sources},\\
 &295{,}480\text{ individual sources, seven direct returns}\\
3^7\leq \alpha\leq3^8&650\text{ closed parameter cells}\\
3^8\leq \alpha\leq3^9&508\text{ cells}\\
3^9\leq \alpha\leq3^{10}&470\text{ cells}\\
3^{10}\leq \alpha\leq3^{11}&552\text{ cells}\\
3^{11}\leq \alpha\leq3^{12}&547\text{ cells}\\
\alpha\geq3^{12}&\text{continuous covers and uniform tail}.
\end{array}
\]
There are exactly \(8{,}033{,}550\) reduced fractions in the first
range. The five closed covers verify every denominator separately
and have no gaps at their joins. The continuous covers for
\(\alpha\geq3^{12}\), described in Appendix~\ref{app:integer-certificates},
cover every real normalized exponent \(t\), with integer \(d\).
For rational \(\alpha=a/b\), taking \(k\) divisible by \(b\) makes the
same transfer valid. Thus their use does not require integrality
of \(\alpha\).

The seven remaining exponents have direct-return certificates:
\[
\begin{array}{c|rrrrrrr}
\alpha&24/23&19/18&13/12&8/7&8/3&11/4&53/2\\ \hline
M&19&17&17&15&17&17&16.
\end{array}
\]
Each certificate specifies a genuine Cantor rectangle, a rational
slope \(r\), an initial level \(T\in B_r\), and a normalization
\((rs+t)/2-T+H(s,t)\), with \(H(0,0)=0\).
It verifies two incomparable returns from every point of \(B_r\),
of length at most \(M\), with margin \(\eta\), and
\[
\sup_{[0,1]^2}(|\partial_sH|+|\partial_tH|)
<\frac{\eta}{2\,3^M}.
\]
Here \(\eta=10^{-7}\) for \(8/3\) and \(\eta=10^{-5}\) otherwise.
The exact recursion and algebraic error bounds are given in
Appendix~\ref{app:real-certificates}. Lemma~\ref{lem:strict-returns}
then gives the lower bound \(\log2/(M\log3)>0\) in each case.
All real exponents \(\alpha>1\) have now been covered.
\end{proof}

For a concrete instance of the direct-return construction, when
\(\alpha=8/3\) take
\[
x=\frac{1349417288366350224+s}{3^{38}},\qquad
y=\frac{5541340632044816+t}{3^{35}},\qquad s,t\in C.
\]
The prefixes contain only ternary digits \(0,2\). Divide
\(x^{8/3}+y^{8/3}-1\) by twice its \(t\)-derivative at \((0,0)\).
For the rational reference slope
\[
r=\frac{5621659580810688002470561}{3884026733719813895964984},
\]
the normalized level lies strictly in the middle component of
\(B_r\). Exact cube-root enclosures give
\[
\sup(|\partial_sH|+|\partial_tH|)
<1.512787\cdot10^{-16}<\frac{10^{-7}}{2\,3^{17}}.
\]
The return certificate covers all three components in at most
seventeen steps, so
\(\dim_{\mathrm H}(E_{8/3,1/3})\geq\log2/(17\log3)\).
The theorem asserts positivity for each exponent; it does not
assert a lower bound uniform in \(\alpha\) or determine the exact dimension.

\subsection{Irrational exponents}\label{subsec:irrational}\label{subsec:universal}

For irrational exponents, the first Hausdorff phase transition occurs
at the universal lower bound \(1/4\). We begin with the dimension
estimate that gives this lower bound for every real exponent. This requires
control of exceptional tangents, rather than an almost-everywhere
slicing statement. We use the following consequence of Shmerkin's
theorem \cite[Theorem~6.2]{Shmerkin2019}: a homogeneous self-similar
probability on the line with four equal weights and exponential
separation has, for every
\[
0<t<\min\{\log4/\log(1/\lambda),1\},
\]
a bound \(\nu([z-r,z+r])\leq A_t r^t\), uniform in \(z,r\).
Exponential separation means that distinct level-\(j\) translations
are at least \(\rho^j\) apart for some \(\rho>0\) and infinitely many
\(j\). The constant \(A_t\) need not be uniform over the systems.

\begin{theorem}\label{thm:dimension-upper}
For every real \(\alpha>1\) and \(0<\lambda\leq1/2\),
\[
\dim_{\mathrm H}(E_{\alpha,\lambda})\leq
\max\left\{\frac{2\log2}{\log(1/\lambda)}-1,0\right\}.
\]
\end{theorem}
\begin{proof}
The case \(\lambda=1/2\) is immediate. Fix \(\lambda<1/2\), and
give each depth-\(j\) square of \(C_\lambda^2\) mass \(4^{-j}\).
Projection by \((x,y)\mapsto y-ax\) gives a four-map homogeneous
measure. Two distinct planar translations at depth \(j\) have a
nonzero coordinate difference of absolute value at least
\((1-2\lambda)\lambda^{j-1}\). Thus, for \(0<\rho<\lambda\), each pair
with different horizontal translations excludes a slope interval
of length at most
\[
\frac{2\lambda}{1-2\lambda}(\rho/\lambda)^j
\]
when its projected separation is less than \(\rho^j\).
Pairs with the same horizontal translation exclude no slopes for
large \(j\). There are at most \(16^j\) pairs. The limsup of these
intervals has Hausdorff dimension at most
\(\log16/\log(\lambda/\rho)\). Failure of exponential separation
belongs to this limsup for every \(\rho\); letting \(\rho\downarrow0\)
shows that the exceptional slope set has dimension zero.

Fix a nonexceptional slope \(a\) and an exponent \(t\) as above.
The projection estimate implies that a strip
\(\lvert y-ax-b\rvert\leq\lambda^j\) meets at most
\begin{equation}
B_{a,t}\,4^j\lambda^{jt}
\label{eq:strip-count}
\end{equation}
depth-\(j\) copies of the dust. Indeed, every such copy projects
inside an interval of radius \((|a|+3)\lambda^j\), and the copies
are disjoint and have equal mass.

Let \(f\) be a \(C^1\) graph and suppose \(f'(x_0)=a\).
Given \(\varepsilon>0\), choose a fixed block length \(j\) with
\(\log\max(B_{a,t},1)/(j\log(1/\lambda))<\varepsilon\).
On a sufficiently short interval \(J\) about \(x_0\),
\(\lvert f'-a\rvert\leq\lambda^j\).
Inside any parent square of side \(\lambda^k\), the graph over
\(J\) lies in a strip of width at most \(2\lambda^{k+j}\)
parallel to this fixed slope. Rescaling the parent and applying
\eqref{eq:strip-count} bounds its descendants \(j\) levels later.
Iteration gives upper box dimension at most
\[
\frac{\log4}{\log(1/\lambda)}-t+\varepsilon
\]
for the intersection over \(J\). A countable subcover of the
nonexceptional graph points suffices for Hausdorff dimension.
Let \(\varepsilon\downarrow0\) and then let \(t\) approach its
allowed supremum.

For \(f(x)=(1-x^\alpha)^{1/\alpha}\), \(0<x<1\),
\[
f''(x)=-(\alpha-1)x^{\alpha-2}(1-x^\alpha)^{1/\alpha-2}<0.
\]
On each compact subinterval, the derivative map is locally
bi-Lipschitz. The preimage of the exceptional slope set therefore
has dimension zero, as does its graph. Adding the two axis points
does not change the bound.
\end{proof}

\begin{corollary}\label{cor:first-lower}
For every real \(\alpha>1\), \(\lambda_{\min}(\alpha)\geq1/4\).
\end{corollary}
\begin{proof}
The preceding bound is zero for \(0<\lambda\leq1/4\).
\end{proof}

We construct parameters as roots of the equation, keeping both
Cantor tails unrestricted. The distinction between the resulting
root distribution and Lebesgue measure will matter.
Write
\[
\pi_\lambda(\omega)=(1-\lambda)\sum_{i=0}^\infty\omega_i\lambda^i,
\qquad \omega\in\{0,1\}^{\mathbb N_0},
\]
and let \(\mu_\lambda\) be its fair Bernoulli distribution.
All parameter derivatives of \(\pi_\lambda\) of a fixed finite
order are uniformly bounded on compact subintervals of \((0,1)\).
For a measure with density \(h\), we use
\(\|h\|_{H^\sigma}^2=\int(1+|\xi|^2)^\sigma|\widehat h(\xi)|^2\,d\xi\).

Here is the regularity input for the linear models. The proof uses
the decomposition into a sparse Fourier-decaying factor and a regular
convolution factor; see~\cite{Shmerkin2014,ShmerkinSolomyak2016}
for this method. The specific inputs are cited below.

\begin{lemma}\label{lem:linear-regularity}
Outside a set of Hausdorff dimension at most one in
\((1/4,1/2)\times(0,\infty)\), the distribution of
\(rX+Y\), for independent \(X,Y\) with distribution \(\mu_\lambda\),
has a density in \(H^\sigma\) for some \(\sigma>0\).
\end{lemma}
\begin{proof}
Work on a compact rectangle \(a\leq\lambda\leq b\),
\(0<c\leq r\leq d\), with \(a>1/4\). Choose an integer \(M\)
such that \(4^{M-1}a^M>1\), and split the paired digits into
positions divisible by \(M\) and the remaining positions.
The latter distribution has contraction \(\lambda^M\) and
\(4^{M-1}\) equal weights. It has an \(L^2\) density outside a
parameter set of dimension at most one by
\cite[Theorem~9.2]{Shmerkin2019}. Its analytic nondegeneracy
holds because varying \(r\) in an identically vanishing address
difference would make both coordinate differences vanish;
strong separation for \(C_\lambda\) rules this out.

The sparse factor has power Fourier decay outside a
zero-dimensional set of \(\lambda\), by
\cite[Proposition~2.3]{Shmerkin2014}, applied to \(\lambda^M\).
Its second coordinate alone supplies that decay; the first factor
has Fourier modulus at most one. Multiplying this decay by the
square-integrable Fourier transform of the other factor gives a
positive Sobolev order. The product of the exceptional contraction
set with a bounded slope interval has dimension at most one.
A countable exhaustion of the parameter domain completes the proof.
\end{proof}

The next lemma supplies the nonlinear step. A cylinder coordinate
means \(a(\lambda)+\lambda^h\pi_\lambda(\omega)\), where \(a\)
is the left endpoint of one fixed binary word of length \(h\).

\begin{lemma}\label{lem:root-fibres}
Fix \(\alpha>1\), two cylinder coordinates \(x_\omega(\lambda)\),
\(y_\tau(\lambda)\) of lengths \(h,\ell\), and
\([a,b]\subset(1/4,1/2)\) on which both coordinates stay positive.
Suppose every paired address has a unique root of
\(x_\omega(\lambda)^\alpha+y_\tau(\lambda)^\alpha=1\) in \([a,b]\),
and the parameter derivative of the residual is bounded away
from zero at all these roots.
The distribution of these roots under fair independent addresses
is absolutely continuous, and for almost every parameter with
respect to that distribution,
\[
\dim_{\mathrm H}(E_{\alpha,\lambda})=\frac{\log4}{\log(1/\lambda)}-1.
\]
\end{lemma}
\begin{proof}
Denote the root map by \(f\), and the fair product address measure
by \(\mu\). Compactness and the implicit function estimate give
uniform local root estimates. Distinct paired depth-\(j\)
cylinders are separated by a constant times \(a^j\) in the
incidence coordinates
\[
(f(\omega,\tau),x_\omega(f(\omega,\tau)),
                       y_\tau(f(\omega,\tau))).
\]
To check this, compare their coordinates at one common parameter:
strong separation gives the stated lower bound, whereas moving
the parameter changes the coordinates by at most a fixed multiple
of its displacement. Consequently this incidence measure is
Frostman with exponent \(\log4/\log(1/a)>1\).

The relevant linear direction is
\[
r=\lambda^{h-\ell}(x/y)^{\alpha-1}.
\]
On the positive compact arc, the state \((\lambda,r)\) recovers
\(x/y=(r\lambda^{\ell-h})^{1/(\alpha-1)}\) and then \(x,y\) uniquely
from \(x^\alpha+y^\alpha=1\). This inverse is Lipschitz. The state measure
therefore gives zero mass to the exceptional set of
Lemma~\ref{lem:linear-regularity}.
Choose a compact address set \(K\) of positive measure on which
the linear densities have one common positive Sobolev order and
bounded norm. Such sets exhaust the source measure: the Sobolev
norm is lower semicontinuous in the state by Fatou's lemma.

For every paired prefix \(w\) of length \(j\) meeting \(K\),
choose a representative in \(K\), with root \(\lambda_w\),
coordinates \(x_w,y_w\), remaining tails \(u_w,v_w\), and residual
parameter derivative \(D_w\). Freeze the complete linear kernel
as the law of
\begin{equation}
\lambda_w-\frac{\alpha x_w^{\alpha-1}\lambda_w^{h+j}}{D_w}(X-u_w)
           -\frac{\alpha y_w^{\alpha-1}\lambda_w^{\ell+j}}{D_w}(Y-v_w),
\label{eq:root-kernel}
\end{equation}
where \(X,Y\) are independent full \(\mu_{\lambda_w}\) tails.
The fresh tails are not restricted to \(K\). Its density \(h_w\)
satisfies
\begin{equation}
\|h_w\|_2^2\leq C\lambda_w^{-j},
\label{eq:kernel-l2}
\end{equation}
and, after division by its scale \(\lambda_w^j\), a uniform
\(H^\sigma\) bound. The coefficients and their ratio stay in
fixed compact sets.

Splitting the fresh tails into their first digits decomposes
\eqref{eq:root-kernel} exactly into four equally weighted kernels
still frozen at the parent. Replace each retained child kernel
by the one at its own representative. The two roots differ by
\(O(\lambda_w^j)\). Shared depth-\(j\) coordinate tails and their
parameter derivatives differ by \(O(\lambda_w^j)\) and
\(O((j+1)\lambda_w^{j-1})\), respectively. Taylor expansion of
the residual therefore bounds the unscaled coupling error by
\(O((j+1)\lambda_w^{2j-1})\). Parent and child scales are uniformly
comparable, since \(j|\lambda_{wi}-\lambda_w|\to0\) uniformly.
After normalization the coupling error is \(O((j+1)b^j)\).

For two bounded-support probability densities with bounded
\(H^\sigma\) norm and coupling displacement at most \(t\),
Fourier splitting gives
\begin{equation}
\|h-\widetilde h\|_1\leq C t^{2\sigma/(3+2\sigma)}.
\label{eq:kernel-interpolation}
\end{equation}
Indeed, below frequency \(R\) the squared Fourier difference
integrates to \(O(t^2R^3)\); above it the Sobolev bounds give
\(O(R^{-2\sigma})\). Optimize \(R\) and use Cauchy--Schwarz on
the common bounded support. Hence the parent--child replacement
errors are bounded by a summable sequence
\(C(j+1)^\kappa b^{\kappa j}\), with \(\kappa>0\).

Let \(K_j\) be the union of depth-\(j\) cylinders meeting \(K\).
Sum the complete kernels with weights \(4^{-j}\). Refinement
consists of exact splitting, deletion of children not meeting
\(K\), and the preceding replacements. Deleted masses telescope
because \(K_j\downarrow K\); replacement errors are summable.
The sums converge in \(L^1\). Direct root linearization identifies
their weak limit as \(f_*(\mu|K)\), which therefore has a density
\(g\).

Repeat the construction inside a fixed prefix \(w\) of length \(j\).
Let \(g_w\) be the density of \(4^j f_*(\mu|K\cap[w])\).
Deletion has negative sign, so summing only positive replacement
errors gives
\begin{equation}
g_w\leq h_w+e_w,\qquad e_w\geq0,\qquad
\|e_w\|_1\leq C(j+1)^\kappa b^{\kappa j}.
\label{eq:root-one-sided}
\end{equation}
For any \(T>0\), splitting according to \(h_w>T/2\) yields
\begin{equation}
\int_{\{g_w>T\}}g_w
\leq \frac{2\|h_w\|_2^2}{T}+3\|e_w\|_1.
\label{eq:root-tail}
\end{equation}
On the support of \(g_w\), \(|\lambda-\lambda_w|\leq C\lambda_w^j\),
so \(\lambda^{-j}/\lambda_w^{-j}\to1\) uniformly.
Apply \eqref{eq:root-tail} at threshold
\(\tfrac12\lambda_w^{-j}e^{\delta j}\), multiply by \(4^{-j}\),
and sum over prefixes. Equations \eqref{eq:kernel-l2}--\eqref{eq:root-tail}
give a summable bound for the event
\(g_w(\lambda)>\lambda^{-j}e^{\delta j}\).

Apply the disintegration theorem~\cite{Kallenberg2021} to
\(\mu|K\), after normalization, over \(f\). For almost every root with
\(0<g(\lambda)<\infty\), the conditional prefix mass is exactly
\[
4^{-j}\frac{g_w(\lambda)}{g(\lambda)}.
\]
Borel--Cantelli gives, at conditional-almost every address,
an eventual bound
\((4\lambda)^{-j}e^{\delta j}/g(\lambda)\).
At this fixed parameter, different paired cylinders are separated
by a fixed multiple of \(\lambda^j\). A ball of a sufficiently
small fixed multiple of that radius about an address lies in its
own cylinder. The lower local dimension of the conditional
measure is therefore at least
\((\log(4\lambda)-\delta)/\log(1/\lambda)\).
Let \(\delta\downarrow0\) through a countable sequence.
Theorem~\ref{thm:dimension-upper} gives the reverse inequality.
Finally exhaust almost all source addresses by the compact sets
\(K\) chosen above. Their pushforwards are absolutely continuous
and give no mass to the failure set, proving both assertions for
the full root law.
\end{proof}

\begin{theorem}\label{thm:irrational-parameters}
For every irrational \(\alpha>1\) and every nonempty open interval
\(J\subset(1/4,1/2)\), the set
\(\{\lambda\in J:\dim_{\mathrm H}(E_{\alpha,\lambda})>0\}\)
has positive Lebesgue measure.
\end{theorem}
\begin{proof}
Choose \(1/4<a<\lambda_0<b<1/2\) with \([a,b]\subset J\), and set
\[
d=\frac{\alpha\log(1-\lambda_0)-\log \alpha}{\log\lambda_0}>0.
\]
The function \((1-\lambda)^\alpha-\alpha\lambda^d\) decreases strictly and
vanishes at \(\lambda_0\). Irrational rotation gives positive
integers \(k_i,m_i\to\infty\) with \(m_i-\alpha k_i\to d\).
For an integer \(L\geq1\), use the actual cylinder coordinates
\[
x=\lambda^{k_i}(1-\lambda+\lambda^L\pi_\lambda(\omega)),
\qquad
y=1-\lambda^{m_i}+\lambda^{m_i+L}\pi_\lambda(\tau).
\]
Their prefixes are \(0^{k_i}10^{L-1}\) and \(1^{m_i}0^L\).
Dividing the residual by \(\lambda^{\alpha k_i}\) and letting
\(i\to\infty\) gives, uniformly in both addresses in \(C^1[a,b]\),
\[
(1-\lambda+\lambda^L\pi_\lambda(\omega))^\alpha
-\alpha\lambda^d(1-\lambda^L\pi_\lambda(\tau)).
\]
The remainder is \(O_\alpha(m_i b^{m_i-1})\) in \(C^1\).
As \(L\to\infty\), this converges to the strictly decreasing
comparison function. Choose \(L\), then \(i\), sufficiently large.
Every tail pair now has a unique transverse root in \((a,b)\).
Lemma~\ref{lem:root-fibres} gives a nonzero absolutely continuous
root law carried by positive-dimensional parameters, proving the
positive Lebesgue measure assertion.
\end{proof}

\begin{theorem}\label{thm:irrational-onset}
For every irrational \(\alpha>1\), \(\lambda_{\min}(\alpha)=1/4\).
\end{theorem}
\begin{proof}
The preceding theorem gives positive-dimensional parameters
arbitrarily close to \(1/4\) from above; Corollary~\ref{cor:first-lower}
gives the opposite inequality.
\end{proof}

\subsection{Rational exponents}
\label{subsec:rational-continuity}

For a reduced rational exponent, the available near-axis phases
form a lattice. Its spacing gives a quantitative bound on the first
Hausdorff phase transition. Together with the irrational case, this
bound determines the continuity of the transition parameter as a
function of the exponent.

\begin{theorem}\label{thm:denominator}
For every reduced rational \(\alpha=a/b\geq2\), with positive integers
\(a,b\),
\[
0\leq\lambda_{\min}(\alpha)-\frac14
\leq\frac{(3/2)^\alpha\log4}{2\alpha b}.
\]
\end{theorem}
\begin{proof}
Set
\[
d_0=\frac{\alpha\log(4/3)+\log \alpha}{\log4},
\qquad d=\frac{\lfloor bd_0\rfloor+1}{b}.
\]
Thus \(0<d-d_0\leq1/b\). The decreasing function
\((1-\lambda)^\alpha-\alpha\lambda^d\) is positive at \(1/4\) and negative
at \(1/2\): the latter follows from
\(d\leq d_0+1<\alpha+\log_2 \alpha\), equivalent to \(\alpha3^\alpha>4\).
Let its unique root be \(\rho\in(1/4,1/2)\).
The integer equation \(bm-ak=bd\) has solutions with
\(k,m\to\infty\). Along them \(m-\alpha k=d\) exactly.
The cylinder construction in the preceding proof therefore puts
a transverse full-tail root family in every neighborhood of
\(\rho\), by first choosing \(L\) and then \(k,m\) large.
Lemma~\ref{lem:root-fibres} gives \(\lambda_{\min}(\alpha)\leq\rho\).

At \(1/4\) the comparison function is at most
\((3/4)^\alpha\log4/b\), by \(1-e^{-u}\leq u\).
The absolute value of its derivative on \([1/4,1/2]\) is at least
\(\alpha2^{1-\alpha}\). The mean value theorem gives the stated upper
bound; Corollary~\ref{cor:first-lower} gives the lower one.
\end{proof}

\begin{theorem}\label{thm:punctured-limit}
For every \(\alpha_0\geq2\),
\[
\lim_{\substack{\alpha\to \alpha_0\\\alpha\geq2,\ \alpha\ne \alpha_0}}\lambda_{\min}(\alpha)=\frac14.
\]
\end{theorem}
\begin{proof}
Irrational exponents have value \(1/4\).
In a bounded interval there are only finitely many reduced rationals
of bounded denominator. Thus denominators along distinct rational
exponents approaching \(\alpha_0\) tend to infinity.
The coefficient of \(1/b\) in Theorem~\ref{thm:denominator} is
locally bounded, proving the limit.
\end{proof}

It follows that the first Hausdorff phase transition parameter on \([2,\infty)\) is continuous
at precisely those exponents where its value is \(1/4\).
Rationality alone does not force a discontinuity.

\begin{proposition}\label{prop:powers-four}
For every integer \(r\geq1\), \(\lambda_{\min}(4^r)=1/4\).
\end{proposition}
\begin{proof}
Fix \(\alpha=4^r\), set \(m=\alpha k+r\), \(L=m+r+2\), and use the cylinders
\[
x=\lambda^k(1-\lambda^L+\lambda^L\pi_\lambda(\omega)),
\qquad
y=1-\lambda^m+\lambda^{m+L}\pi_\lambda(\tau).
\]
For fixed \(\alpha,r\), their residual divided by \(\lambda^{\alpha k}\)
converges uniformly in \(C^1\), as \(k\to\infty\), to
\(1-\alpha\lambda^r\) on every compact subinterval of \((0,1/2)\).
The curvature remainder is \(O_\alpha(\lambda^m)\), with derivative
\(O_\alpha(m\lambda^{m-1})\); the tails contribute
\(O_\alpha(\lambda^L)\) and \(O_\alpha(L\lambda^{L-1})\).

At \(\lambda=1/4\), first take both normalized tails at the
endpoint values giving \(x=\lambda^k,y=1-\lambda^m\).
The leading terms cancel since \(\alpha4^{-r}=1\). Taylor's integral
remainder makes the remaining normalized residual at least
\[
\frac{\alpha-1}{2}(1-4^{-m})^{\alpha-2}4^{-m}
>\frac{\alpha-1}{4}4^{-m}.
\]
Here \((\alpha-2)4^{-m}\leq4^{-\alpha}\leq1/256\).
Varying the first tail loses at most
\(\alpha4^{-L}=4^{-m}/16\); varying the second tail only increases
the residual. Thus every tail pair has positive residual at \(1/4\).
For any fixed \(c\in(1/4,1/2)\), uniform \(C^1\) convergence gives
negative residual at \(c\) and negative derivative throughout
\([1/4,c]\), for all sufficiently large \(k\).
Every pair has a unique root in \((1/4,c)\); for fixed \(k\),
compactness keeps all roots in a compact subinterval.
Lemma~\ref{lem:root-fibres} gives positive-dimensional parameters
there. Let \(c\downarrow1/4\) and use Corollary~\ref{cor:first-lower}.
\end{proof}

%%%%%%%%%%%%%%%%%%%%%%%%%%%%%%%%%%%%%%%%%%%%%%%%%%%%%%%%%%%%%%%%%%%
\section{Hausdorff phase transition diagram}\label{sec:last}

We now fix the exponent and develop the Hausdorff phase transition
diagram as the contraction ratio varies. In
Section~\ref{subsec:last-upper}, we prove the uniform bound
\(\lambda_{\max}(\alpha)<(\sqrt3-1)/2\) for every real \(\alpha>1\).
Thus the solution set has positive Hausdorff dimension throughout
\([(\sqrt3-1)/2,1/2]\), and this region extends to the left by
an amount that may depend on \(\alpha\). For \(\alpha=2\), we also separate the
first and last transitions by exhibiting zero-dimensional intervals
after positive dimension first appears.

In Section~\ref{sec:ae}, for each fixed irrational \(\alpha>1\), we prove
\[
\dim_{\mathrm H}(E_{\alpha,\lambda})
=\frac{2\log2}{\log(1/\lambda)}-1
\]
for Lebesgue almost every \(\lambda\in(\sqrt3/4,1/2)\).
This establishes the expected dimension formula in an explicit
parameter range. The exact dimension of the unit-circle intersection
at \(\lambda=1/3\), asked for in~\cite{Jiang2026}, remains open.

\subsection{Regions of positive Hausdorff dimension}\label{subsec:last-upper}

We establish the uniform bound by constructing strict interior
levels whose intersections with the Cantor dust have positive dimension.

The following criterion gives two intervals of derivative ratios for
which every interior level has positive dimension.

\begin{lemma}\label{lem:two-bands}
Let \(0<\lambda<1/2\), and let \(f,g\) be increasing \(C^1\)
functions on \([0,1]\) with positive derivatives.
If all ratios \(g'(t)/f'(s)\), \(s,t\in[0,1]\), are greater than
one and belong to a compact subinterval of either
\begin{equation}\label{eq:two-bands}
\left(\frac1\lambda-1-\lambda,\frac1{1-\lambda}\right)
\quad\text{or}\quad
\left(\frac1\lambda-1,\frac1{1-\lambda-\lambda^2}\right),
\end{equation}
then every strict interior level of \(f+g\) on \([0,1]^2\)
meets \(C_\lambda^2\) in a set of positive Hausdorff dimension.
\end{lemma}
\begin{proof}
Subtract the values at zero. Choose strict bounds
\(1<L\leq g'(t)/f'(s)\leq U\) in one of the intervals
\eqref{eq:two-bands}. The four child image hulls, ordered as
\(00,10,01,11\), are
\begin{gather*}
[0,f(\lambda)+g(\lambda)],\quad
[f(1-\lambda),f(1)+g(\lambda)],\\
[g(1-\lambda),f(\lambda)+g(1)],\quad
[f(1-\lambda)+g(1-\lambda),f(1)+g(1)].
\end{gather*}
Consecutive hulls overlap strictly: the outer overlaps follow from
\(L>1/\lambda-2\), and the middle one from
\(U<1/(1-2\lambda)\). A level in no two child interiors must
belong to the respective closed regions
\begin{gather*}
[0,f(1-\lambda)],\quad
[f(\lambda)+g(\lambda),g(1-\lambda)],\\
[f(1)+g(\lambda),f(1-\lambda)+g(1-\lambda)],\quad
[f(\lambda)+g(1),f(1)+g(1)].
\end{gather*}
Empty regions are discarded. Restrict to the continuing child,
subtract its lower-corner value, and divide by \(\lambda\).
The derivative bounds are preserved. The inequalities
\begin{align*}
f(\lambda)+g(\lambda)&>f(1-\lambda^2),\\
g(1-\lambda)&<f(1)+g(\lambda^2),\\
f(1-\lambda)&<f(\lambda^2)+g(\lambda)
\end{align*}
follow respectively from
\(L>1/\lambda-1-\lambda\),
\(U<1/(1-\lambda-\lambda^2)\), and the first inequality on \(L\).
They show that a continuing \(10\) must be followed by \(10\),
and that \(00\) cannot be followed by \(11\), as long as the level
avoids strict overlap. Simultaneous reflection gives the corresponding
statements for \(01\) and \(11\).

The constant \(10\) address has level \(f(1)\). This is outside
its closed region: in the first band \(g(1-\lambda)<f(1)\),
whereas in the second band \(f(\lambda)+g(\lambda)>f(1)\).
Reflection excludes the constant \(01\) address. Thus an infinite
address avoiding overlap can only be \((00)^\infty\) or
\((11)^\infty\), at the two extreme levels. Every strict interior
level reaches a strict overlap in finitely many steps.

We need uniform returns to obtain positive dimension. Normalize
initially by \(\min f'=1\); then
\(1\leq f'\leq U/L\) and \(L\leq g'\leq U\).
The closed family of Lipschitz pairs with these individual increment
bounds and cross-variable increment-ratio bounds \([L,U]\) is
compact in the uniform topology. Translated cylinder restrictions,
divided by their geometric scale, preserve this whole family;
their minimum first derivative need not equal one.
All preceding comparisons use increment inequalities and hence
remain valid on the closed family.

Choose \(0<\delta<\lambda(1+L)/(1+\lambda)\).
Use \([\delta,f(1)+g(1)-\delta]\) as the central level domain.
The images of central domains under all-zero prefixes are
\[
[\delta\lambda^j,
 f(\lambda^j)+g(\lambda^j)-\delta\lambda^j],\qquad j\geq0.
\]
Successive intervals overlap strictly by the lower increment bounds
and the choice of \(\delta\). Their interiors reach every positive
level below the upper endpoint of the original central domain;
reflection reaches the remaining interior levels. At a strict overlap,
both children can therefore be brought strictly into their central
domains by finite prefixes. These are two incomparable returning words.
The space of function pairs and their central levels is compact,
and each such pair of strict returns remains valid on an open
neighbourhood. A finite subcover gives a maximum return length
\(M\) and a positive return margin. Iteration produces a separated
binary tree of product cylinders. The cylinder separation
\(1-2\lambda>0\) and the mass distribution principle give dimension
at least \(\log2/(M|\log\lambda|)>0\). An arbitrary interior level
first enters a central domain. This initial finite restriction
does not change positivity of dimension.
\end{proof}

We apply this criterion to actual rectangles and to endpoint models.
Let \([u_-,u_+]\), \([v_-,v_+]\) be positive cylinder hulls of
depths \(p,q\). An actual rectangle satisfies
\begin{equation}\label{eq:uniform-actual}
u_-^\alpha+v_-^\alpha<1<u_+^\alpha+v_+^\alpha,
\end{equation}
and its derivative-ratio interval is
\[
\left[\lambda^{q-p}(v_-/u_+)^{\alpha-1},
      \lambda^{q-p}(v_+/u_-)^{\alpha-1}\right].
\]
An endpoint model, with integers \(j\) and \(k\geq0\), satisfies
\begin{equation}\label{eq:uniform-endpoint}
u_-^\alpha<\alpha\lambda^{j-k\alpha}v_+,
\qquad u_+^\alpha>\alpha\lambda^{j-k\alpha}v_-.
\end{equation}
After reflection of the second coordinate, the power form
\(u^\alpha-\alpha\lambda^{j-k\alpha}v\) has derivative-ratio interval
\[
\left[\lambda^{j-k\alpha+q-p}u_+^{1-\alpha},
      \lambda^{j-k\alpha+q-p}u_-^{1-\alpha}\right],
\]
whereas the zero-equivalent logarithmic form
\(\alpha\log u-\log v-\log \alpha-(j-k\alpha)\log\lambda\) has
first-to-second derivative-ratio interval
\begin{equation}\label{eq:uniform-log-ratio}
\left[\alpha\lambda^{p-q}v_-/u_+,
      \alpha\lambda^{p-q}v_+/u_-\right].
\end{equation}
In each case we may take the reciprocal interval by exchanging
coordinates. A source is accepted when its whole ratio interval
lies strictly above one and inside one band of \eqref{eq:two-bands}.

\begin{theorem}\label{thm:last-real}
For every real \(\alpha>1\),
\[
\lambda_{\max}(\alpha)<\frac{\sqrt3-1}{2}.
\]
\end{theorem}
\begin{proof}
We prove positivity for every
\((\sqrt3-1)/2\leq\lambda<1/2\), and then use strictness at the
left endpoint. The proof combines Lemma~\ref{lem:two-bands} with
exact interval certificates; their predicates, finite parameter
covers, and replay files are given in
Appendix~\ref{app:uniform-certificates}.

First, an accepted endpoint model gives an actual Fermat intersection
for every real exponent. Choose positive integers \(t_i\) tending
to infinity with \(\operatorname{dist}(\alpha t_i,\mathbb Z)\to0\),
and nearest integers \(r_i\) to \(\alpha t_i\). For rational \(\alpha\)
take denominator multiples, and for irrational \(\alpha\) use recurrence
modulo one. With
\[
x=\lambda^{k+t_i}u,\qquad y=1-\lambda^{j+r_i}v,
\]
both coordinates belong to \(C_\lambda\) for all sufficiently
large \(i\). Taylor's formula gives, in \(C^1\) on the fixed rectangle,
\[
\lim_{i\to\infty}
\lambda^{-\alpha(k+t_i)}(x^\alpha+y^\alpha-1)
=u^\alpha-\alpha\lambda^{j-k\alpha}v.
\]
For a logarithmic source we instead use the exactly zero-equivalent
expression
\[
\alpha\log u-
\log\left(\frac{1-(1-\lambda^{j+r_i}v)^\alpha}
                   {\lambda^{\alpha(k+t_i)}}\right).
\]
Its \(C^1\) limit is the certified logarithmic form.
The coordinates stay bounded away from zero, so the logarithm and
its derivative are controlled. The strict corner and derivative
inequalities persist for a sufficiently large fixed \(i\).
Lemma~\ref{lem:two-bands} applies to that actual separable function,
and the fixed affine coordinate map preserves Hausdorff dimension.
The logarithmic ratio is used only with the logarithmic expression.

For \(1<\alpha\leq6/5\), Appendix~\ref{app:uniform-near-one}
constructs explicit endpoint coefficient intervals. Their widths
contain a point of the coefficient grid for all sufficiently large
rational denominators; irrational exponents use density. The
remaining small denominators form an explicitly checked finite list.
Only \(7/6\) and \(6/5\) require actual-rectangle covers, with
\(67\) and \(74\) records, respectively. This proves the whole
interval \(1<\alpha\leq6/5\), including its noncompact endpoint at one.

For \(6/5\leq \alpha\leq500\), the following closed rational rectangle
covers verify \eqref{eq:uniform-actual} or
\eqref{eq:uniform-endpoint}, together with the appropriate strict
band inequalities, throughout the indicated parameter ranges:
\[
\begin{array}{c|r|r}
\text{Exponent interval}&\text{Source rectangles}&\text{Checked subrectangles}\\\hline
{[6/5,2]}&152&165\\
{[2,10]}&1905&3513\\
{[10,100]}&13369&18693\\
{[100,250]}&21855&31840\\
{[250,500]}&51959&53067
\end{array}
\]
These are interval covers, not parameter samples. For example,
the logarithmic corner residual in \eqref{eq:uniform-endpoint}
has exponent derivative
\[
\log u+k\log\lambda-1/\alpha<0,
\]
and the logarithm of a power-model ratio has exponent derivative
\(-k\log\lambda-\log u\geq0\). The logarithmic ratio is
proportional to \(\alpha\). Actual corner sums decrease with \(\alpha\),
and the logarithm of an actual ratio is affine in \(\alpha\).
Consequently exact endpoint bounds control every real exponent
in each rectangle. Rational interval arithmetic controls its
entire contraction interval. An independent sweep of all exponent
boundary events verifies that the active contraction intervals
cover \([(\sqrt3-1)/2,1/2]\) without gaps.

For \(\alpha\geq500\), write \(\alpha=\lambda^{-m}t\) with
\(1\leq t<1/\lambda\). Appendix~\ref{app:uniform-tail} verifies
finite cycles of logarithmic source intervals covering the whole
coefficient phase circle. If \(\lambda^{-m}\geq1000\), there are
\(36\) contraction slabs with \(33755\) sources. The remaining
scales can only have \(m=6,7,8,9\); another \(75\) slabs with
\(22859\) sources cover them. The finite-scale comparison in that
appendix is uniform in \(m\), so this establishes an unbounded
real exponent interval. The endpoint transfer above completes
the positivity proof for every \(\alpha>1\).

At \(\lambda=1/2\), the Cantor set is \([0,1]\) and the Fermat
arc has dimension one. Finally, for each fixed \(\alpha>1\), an accepted
source at \(\lambda=(\sqrt3-1)/2\) gives a sufficiently deep
actual rectangle with finitely many strict corner and derivative
inequalities. They persist on a neighbourhood of this contraction
parameter. Combining this neighbourhood with the full interval
already proved gives positivity above a parameter strictly less
than \((\sqrt3-1)/2\), and hence the claimed strict bound.
The neighbourhood may depend on \(\alpha\).
\end{proof}

The uniform bound identifies a common interval of persistent
positive dimension. It remains to understand whether positive dimension
persists from its first appearance. We now examine this question for
the unit circle, \(\alpha=2\).

\begin{theorem}\label{thm:circle-positive-measure}
For every \(0<\varepsilon<1/2-(2-\sqrt3)\), the set
\[
\{\lambda\in(2-\sqrt3,2-\sqrt3+\varepsilon):
                       \dim_{\mathrm H}(E_{2,\lambda})>0\}
\]
has positive Lebesgue measure.
\end{theorem}
\begin{proof}
For an integer \(k\geq2\), consider the two genuine Cantor cylinders
\[
x=(1-\lambda)\lambda^k+\lambda^{3k+1}\pi_\lambda(\omega),
\qquad
y=1-\lambda^{2k+1}+\lambda^{4k+1}\pi_\lambda(\tau).
\]
Dividing \(x^2+y^2-1\) by \(\lambda^{2k}\) gives
\begin{align*}
&-(4\lambda-1-\lambda^2)
 +2\lambda^{2k+1}\bigl((1-\lambda)\pi_\lambda(\omega)
                                      +\pi_\lambda(\tau)\bigr)\\
&\hspace{12mm}
 +\lambda^{4k+2}\pi_\lambda(\omega)^2
 +\lambda^{2k+2}
             \bigl(1-\lambda^{2k}\pi_\lambda(\tau)\bigr)^2 .
\end{align*}
At \(\lambda=2-\sqrt3\) this is strictly positive, uniformly
over the addresses for each fixed \(k\). At \(b=7/25\) it is
strictly negative: the positive terms are at most
\(2b^5(2-b)+b^{10}+b^6\), whereas
\[
4b-1-b^2-2b^5(2-b)-b^{10}-b^6>0.035.
\]
On this parameter interval, differentiating the coding series
gives \(|\partial_\lambda\pi_\lambda|\leq3\).
The absolute derivative of the three positive terms is at most
\[
2(10b^4+7b^5)+10b^9+6b^{10}+6b^5
                         +2b^6(4b^3+3b^4)<0.158 .
\]
Here the bounds for \(k=2\) dominate all \(k\geq2\), as follows
by comparing consecutive powers and their integer factors.
The derivative of the entire residual therefore lies in
\((-4,-3)\). Every address has a unique uniformly transverse
root, and the mean value theorem gives
\[
0<\lambda-(2-\sqrt3)<2b^{2k+1}.
\]
For each fixed \(k\), compactness places all roots in a closed
subinterval of \((2-\sqrt3,2-\sqrt3+2b^{2k+1})\).
Lemma~\ref{lem:root-fibres} applies to these two fixed cylinders.
Their root law is an absolutely continuous probability measure,
and almost every root for that law has positive-dimensional
fibre. Choosing \(k\) sufficiently large proves the assertion.
\end{proof}

\begin{theorem}\label{thm:circle-onset}
The first Hausdorff phase transition for the unit circle is
\(\lambda_{\min}(2)=2-\sqrt3\).
\end{theorem}
\begin{proof}
We prove that the only solutions for \(0<\lambda\leq2-\sqrt3\)
are the two axis points. Both coordinates cannot lie in the
left first-level interval, since \(2\lambda^2<1\), or both in
the right interval, since \(2(1-\lambda)^2>1\).
After interchanging them, a nonaxis solution must have
\[
x=\lambda^k X,\qquad 1-y=\lambda^mY,\qquad
k,m\geq1,\quad X,Y\in(1-\lambda)+\lambda C_\lambda.
\]
The circle equation becomes
\[
X^2=\lambda^{m-2k}Y(2-\lambda^mY).
\]
If \(m-2k\leq0\), its right-hand side is at least
\((1-\lambda)(2-\lambda)>1\).
If \(m-2k\geq1\), it is strictly smaller than \(2\lambda\),
which is at most \((1-\lambda)^2\leq X^2\).
Both alternatives are impossible. This proves the lower bound;
Theorem~\ref{thm:circle-positive-measure} proves the reverse bound.
\end{proof}

\begin{theorem}\label{thm:circle-zero-intervals}
There is a sequence of nonempty open intervals
\(I_j\subset(2-\sqrt3,1/2)\) with
\(\lim_{j\to\infty}\sup I_j=2-\sqrt3\), on each of which
\(E_{2,\lambda}=\{(0,1),(1,0)\}\).
\end{theorem}
\begin{proof}
For each integer \(j\geq3\), let the endpoints of \(I_j\)
be the roots in \((2-\sqrt3,67/250)\) of
\[
4\lambda-1-\lambda^2=\frac{379}{100}\lambda^{2j+2},
\qquad
4\lambda-1-\lambda^2=\frac{381}{100}\lambda^{2j+2}.
\]
These roots exist and are unique. Indeed, the difference of the
two sides is negative at \(2-\sqrt3\), positive at \(67/250\),
and has derivative at least
\(4-2\lambda-(381/100)8\lambda^7>0\).
The same derivative bound holds for all \(j\geq3\), because
\((2j+4)\lambda^{2j+3}/((2j+2)\lambda^{2j+1})<1\).
The second root is larger than the first.
The defining equation also shows that both roots tend to
\(2-\sqrt3\).

Fix \(\lambda\in I_j\). A putative nonaxis solution again has
\(x=\lambda^kX\), \(1-y=\lambda^mY\), with \(X,Y\) in the
right first-level Cantor cylinder. On the larger container
\([2679/10000,67/250]\), the estimates in the preceding proof
exclude \(m-2k\leq0\); \(m-2k\geq2\) is excluded by
\(2\lambda^2<(1-\lambda)^2\).
Thus \(m=2k+1\), and
\[
X^2-2\lambda Y+\lambda^{2k+2}Y^2=0.
\]
Its minimum on the two outer intervals is
\(\lambda^{2k+2}-(4\lambda-1-\lambda^2)\).
Since \(4\lambda-1-\lambda^2<(381/100)\lambda^{2j+2}\)
and \((381/100)\lambda^2<1\), this excludes \(k<j\).

The same equation forces the more precise cylinders
\[
X=1-\lambda+\lambda^{2j+1}u,\qquad
Y=1-\lambda^{2j}v,\qquad u,v\in C_\lambda.
\]
For if \(X\) is outside the displayed leftmost subcylinder of
the outer Cantor interval, then
\(X\geq(1-\lambda)(1+\lambda^{2j})\), and
\(X^2>(1-\lambda)^2+\lambda^{2j}>2\lambda\).
If \(Y\) is outside the displayed rightmost subcylinder, then
\(1-Y\geq(1-\lambda)\lambda^{2j-1}\), and
\(2\lambda Y<2\lambda-\lambda^{2j}<(1-\lambda)^2\).
The inequalities use
\(2(1-\lambda)^2>1\), \(2(1-\lambda)>1\), and
\((381/100)\lambda^2<1\).

Put \(\tau=(4\lambda-1-\lambda^2)/\lambda^{2j+2}\);
thus \(379/100<\tau<381/100\).
Substitution gives the exact identity
\[
\tau=\frac2\lambda\bigl((1-\lambda)u+v\bigr)
       +\lambda^{2j}u^2
       +\lambda^{2(k-j)}(1-\lambda^{2j}v)^2.
\]
Consequently \((1-\lambda)u+v\) lies between
\[
\frac\lambda2(\tau-\lambda^{2(k-j)})-\frac{\lambda^7}{2}
\quad\hbox{and}\quad
\frac\lambda2(\tau-\lambda^{2(k-j)})+\lambda^7.
\]
The projection \((1-\lambda)C_\lambda+C_\lambda\) has the gap
\[
\bigl(\lambda(2-\lambda),(1-\lambda)^2\bigr).
\]
Inside its cylinder with paired first digits \(00\) and paired
second digits \(11\), it also has the gap
\[
\bigl(\lambda(1-\lambda)(2-\lambda)+\lambda^3(2-\lambda),
      \lambda(1-\lambda)(2-\lambda)+\lambda^2(1-\lambda)^2\bigr).
\]
The latter is a gap of the full projection: it lies below every
other first-level image and above every other second-level image
within the \(00\) image. These assertions follow from
\[
\lambda(1-\lambda)(2-\lambda)+\lambda^2(1-\lambda)^2
 <(1-\lambda)^2,\qquad
(1-\lambda)^2>\lambda(1-\lambda)(2-\lambda).
\]
For \(k\geq j+1\), the candidate interval lies strictly in the
first gap, because
\begin{align*}
\frac\lambda2(379/100-\lambda^2)-\lambda^7/2
   &>\lambda(2-\lambda),\\
381\lambda/200+\lambda^7&<(1-\lambda)^2.
\end{align*}
For \(k=j\), it lies strictly in the second gap, because
\begin{align*}
279\lambda/200-\lambda^7/2
 &>\lambda(1-\lambda)(2-\lambda)+\lambda^3(2-\lambda),\\
281\lambda/200+\lambda^7
 &<\lambda(1-\lambda)(2-\lambda)+\lambda^2(1-\lambda)^2.
\end{align*}
All these polynomial inequalities hold throughout the rational
container specified above; Appendix~\ref{app:certificates}
gives their exact verification. Thus no scale \(k\) permits a
nonaxis solution.
\end{proof}

The zero-dimensional intervals above the first Hausdorff phase
transition give the following separation.
\begin{theorem}\label{thm:known-separation}
For the unit circle,
\[
\lambda_{\min}(2)=2-\sqrt3<\lambda_{\max}(2).
\]
\end{theorem}
\begin{proof}
Theorem~\ref{thm:circle-onset} gives the equality. Every parameter
in any interval supplied by Theorem~\ref{thm:circle-zero-intervals}
is greater than \(2-\sqrt3\) and has a zero-dimensional intersection.
The definition of \(\lambda_{\max}(2)\) gives the strict inequality.
\end{proof}

Thus the unit circle has positive-dimensional intersections followed
by zero-dimensional intersections, before positive dimension persists
near \(1/2\). This establishes non-monotonicity without requiring
the Hausdorff phase transitions to be isolated.
Together with Theorem~\ref{thm:punctured-limit}, the exact value
\(\lambda_{\min}(2)=2-\sqrt3\) also shows that the first Hausdorff
phase transition parameter is discontinuous as a function of the
exponent at \(\alpha=2\).

\subsection{Almost-everywhere dimension}\label{sec:ae}

Lemma~\ref{lem:root-fibres} gives equality for almost every parameter
under a root law. Absolute continuity alone does not turn that into
Lebesgue almost everywhere on an interval. We obtain the stronger
quantifier by constructing small root families near a density point,
with quantitative control of their densities.

The next estimate uses transversality and parameter-averaged Fourier
bounds, methods developed for Bernoulli convolutions and parameterised
projections in~\cite{Solomyak1995,PeresSchlag2000}. We give the proof
for the digit distribution and Sobolev range needed here.

\begin{lemma}\label{lem:average-sobolev}
Let \([a,b]\subset(3/8,1/2)\) be sufficiently short that
\[
0<2\sigma<\frac{2\log(8/3)}{\log(b/a^3)}-1.
\]
On every compact subinterval of \((a,b)\), the squared
\(H^\sigma\) norm of \(\mu_\lambda*\mu_\lambda\) is integrable
with respect to \(\lambda\).
\end{lemma}
\begin{proof}
Apart from the factor \(1-\lambda\), this self-convolution is
the distribution of \(\sum_{j\geq0}D_j\lambda^j\), where
\(D_j=0,1,2\) with probabilities \(1/4,1/2,1/4\).
For a normalized difference series
\(f(x)=1+\sum_{i\geq1}u_i x^i\), \(|u_i|\leq2\),
\begin{equation}
f'(x)-\frac2x f(x)
\leq-\frac2x+2+\frac{2x^2}{(1-x)^2}<0
\quad(0<x<1/2).
\label{eq:transversality}
\end{equation}
The last sign follows after multiplication by \(x(1-x)^2/2\):
the resulting polynomial \(2x^3-3x^2+3x-1\) increases to zero
at \(1/2\). Thus \(f'\) is uniformly negative wherever \(|f|\)
is sufficiently small on \([a,b]\).

Fix two digit sequences whose first difference is at \(j\), and
let \(Q(\lambda)\) be their actual difference, including the
factor \(1-\lambda\). It is a nonzero digit difference times
\((1-\lambda)\lambda^j f(\lambda)\).
For a sufficiently small fixed \(\varepsilon>0\), on the band
\begin{equation}
|Q(\lambda)|<\frac{\varepsilon a^j}{j+1}
\label{eq:small-value-band}
\end{equation}
we have
\[
|Q'|\geq c a^j,\qquad |Q''|\leq C(j+1)^2b^j,
\]
and \(Q'\) has one fixed sign there. Indeed, the derivative of
the prefactor contributes at most \(C(j+1)|f|\), which is made
smaller than the strict bound from \eqref{eq:transversality}
by \eqref{eq:small-value-band}. This factor \(1/(j+1)\) is needed
for a bound uniform in \(j\).

Let \(\psi\) be smooth with compact support in \((a,b)\), and
let \(\Phi\geq0\) be a smooth frequency cutoff supported away
from zero. Its inverse Fourier transform \(h\) is Schwartz and
\(\int h=0\). We claim that for \(R\geq1\)
\begin{equation}
\left|\int\psi(\lambda)R h(RQ(\lambda))\,d\lambda\right|
\leq C\min\{R,\,(j+1)^C R^{-1}(b/a^3)^j\}.
\label{eq:pair-annulus}
\end{equation}
The first bound is immediate. For the second, insert a smooth
cutoff in \eqref{eq:small-value-band}. The fixed derivative sign
prevents leaving and reentering this band in the opposite
direction. Changing variables \(u=Q(\lambda)\) gives a compactly
supported inverse density
\(\psi(\lambda(u))\chi(u)/|Q'(\lambda(u))|\).
Its derivative is at most \(C(j+1)^C(b/a^3)^j\), by the two
derivative bounds and the cutoff scale.
Subtract its value at zero and use \(\int h=0\); this gives the
second bound in \eqref{eq:pair-annulus} on the cutoff part.
Off that part, \(|Q|\geq c a^j/(j+1)\), and the Schwartz bound
of order two gives \(C(j+1)^2R^{-1}a^{-2j}\).
When \(Ra^j/(j+1)<1\), the trivial \(CR\) bound is no larger
than the same expression. This proves \eqref{eq:pair-annulus}.

For two independent digit sequences, the probability of agreement
at one position is \(3/8\). Fourier inversion and Fubini therefore
bound the parameter-averaged frequency annulus by
\[
C\sum_{j\geq0}(3/8)^j
\min\{R,(j+1)^C R^{-1}(b/a^3)^j\}.
\]
Split at the integer nearest \(2\log R/\log(b/a^3)\).
The two geometric sums are at most
\[
C(1+\log R)^C
R^{-\{2\log(8/3)/\log(b/a^3)-1\}}.
\]
Multiplying by \(R^{2\sigma}\) and summing over dyadic \(R\)
proves the assertion. The bounded low frequencies cause no issue.
\end{proof}

By Fourier Cauchy--Schwarz and a change of variables, for \(r\)
in a fixed compact subset of \((0,\infty)\),
\begin{equation}
\|\mu_\lambda*(S_r\mu_\lambda)\|_{H^\sigma}
\leq C M(\lambda),\qquad
M(\lambda)^2:=1+\|\mu_\lambda*\mu_\lambda\|_{H^\sigma}^2.
\label{eq:common-sobolev-bound}
\end{equation}
Here \(S_r\) is multiplication by \(r\). Indeed, the squared norm
is the integral of
\[
(1+|\xi|^2)^\sigma|\widehat\mu_\lambda(\xi)|^2
|\widehat\mu_\lambda(r\xi)|^2,
\]
bounded by the geometric mean
of the two self-convolution integrals.
The same function \(M\) controls every such slope.

\begin{theorem}\label{thm:ae-dimension}
For every fixed irrational \(\alpha>1\), for Lebesgue almost every
\(\lambda\in(\sqrt3/4,1/2)\),
\[
\dim_{\mathrm H}(E_{\alpha,\lambda})=\frac{2\log2}{\log(1/\lambda)}-1.
\]
\end{theorem}
\begin{proof}
Choose a short interval \([a,b]\) inside
\((\sqrt3/4,1/2)\). Put \(d=\log2/\log(1/a)\).
It can be chosen, together with \(\sigma>0\), so that
\begin{equation}
b^2/a<1,\qquad
2\sigma<\frac{2\log(8/3)}{\log(b/a^3)}-1,\qquad
(b^2/a)^{2\sigma}<b^{1-d}.
\label{eq:ae-exponents}
\end{equation}
As \(a,b\to\lambda_0\), the permitted inequalities for \(2\sigma\)
are
\[
1-\frac{\log2}{\log(1/\lambda_0)}
<2\sigma<
\frac{\log(8/3)}{\log(1/\lambda_0)}-1.
\]
They are compatible exactly when \(\lambda_0>\sqrt3/4\).
Shrink inside \((a,b)\) so Lemma~\ref{lem:average-sobolev}
gives a locally integrable \(M^2\) on the whole neighborhood
used below.

Fix a point \(\lambda_0\) there at which \(M^2\) has finite
Lebesgue value. For integers \(L\to\infty\), put
\(\rho=\lambda_0^L\).
Irrationality supplies arbitrarily large positive \(k,m\) with
\[
\left|m-\alpha k+\frac{\log \alpha}{\log\lambda_0}\right|\leq\rho^2,
\qquad m^2b^m\leq\rho^3,\qquad k\geq L.
\]
Take the genuine cylinders
\begin{equation}
x=\lambda^k(1-\lambda^L+\lambda^L\pi_\lambda(\omega)),
\qquad y=1-\lambda^m+\lambda^{m+L}\pi_\lambda(\tau).
\label{eq:centered-cylinders}
\end{equation}
Their residual divided by \(\lambda^{\alpha k}\) is
\begin{equation}
(1-\lambda^L+\lambda^L\pi_\lambda(\omega))^\alpha
-\alpha\lambda^{m-\alpha k}(1-\lambda^L\pi_\lambda(\tau))
+O_\alpha(\lambda^m).
\label{eq:centered-residual}
\end{equation}
The remainder and its first two parameter derivatives are
uniformly negligible by the choice of \(m\).
The limiting function \(1-\alpha\lambda^{-\log \alpha/\log\lambda_0}\)
vanishes transversely at \(\lambda_0\), with negative derivative.
On a window of radius \(C\rho\), the residual's first derivative
is bounded away from zero, its second derivative is bounded,
and its two tail derivatives are positive and comparable to
\(\rho\). Parameter derivatives of tail coefficients introduce
only polynomial factors in \(L\), since \(L^2\rho\to0\).
Choosing \(C\) sufficiently large gives opposite endpoint signs
for all tails. Hence all pairs have unique roots in this window.
Call their map \(f\), and their probability law \(\nu\).

Shared depth-\(j\) prefixes give root displacement at most
\(C\rho b^j\). There is also the small-ball estimate
\begin{equation}
\nu(I)\leq C(|I|/\rho)^d.
\label{eq:root-small-ball}
\end{equation}
To verify it, fix the second address. At a common parameter,
two first addresses first differing at position \(j\) are
separated by at least \(c a^j\). The positive first-tail
derivative and bounded parameter derivative of
\eqref{eq:centered-residual} separate their roots by
\(c\rho a^j\). An interval smaller than this meets at most one
first-coordinate cylinder of that depth, of probability \(2^{-j}\).
Choose \(j\) according to its length and integrate over the
second address.

For each paired prefix \(w\) of depth \(j\), let \(I_w\) be an
interval of length \(C\rho b^j\) containing its full root image,
centered at one of its roots. Equation \eqref{eq:root-small-ball}
implies
\begin{equation}
\sum_{|w|=j}4^{-j}1_{I_w}(u)\leq Cb^{jd}.
\label{eq:root-overlap}
\end{equation}
Choose an auxiliary freezing point \(\lambda_w\in I_w\) such that
\[
M(\lambda_w)^2\leq\frac2{|I_w|}\int_{I_w}M(\lambda)^2\,d\lambda.
\]
Summing with \eqref{eq:root-overlap}, and using that \(\lambda_0\)
is a finite Lebesgue point, gives
\begin{equation}
\sum_{|w|=j}4^{-j}M(\lambda_w)^2
\leq C_{\lambda_0}b^{-j(1-d)}
\label{eq:average-freezing}
\end{equation}
uniformly as \(\rho\to0\).

Set the two remaining tails in the residual to zero and denote
the resulting function by \(G_w(\lambda)\).
Let \(A_w,B_w\) be its derivatives with respect to the two
remaining tail values, including their factors \(\lambda^j\).
The complete affine kernel at \(\lambda_w\) is the law of
\begin{equation}
\lambda_w-\frac{G_w(\lambda_w)+A_w(\lambda_w)X+
 B_w(\lambda_w)Y}{G'_w(\lambda_w)},
\label{eq:newton-kernel}
\end{equation}
with independent full \(\mu_{\lambda_w}\) tails.
The Newton constant is necessary because \(\lambda_w\) need
not itself be a root. We have
\(-G'_w\asymp1\), \(A_w,B_w\asymp\rho\lambda_w^j\).
Taylor expansion under the same-address coupling bounds the
distance from the actual root by
\begin{equation}
C\rho^2(L+j+1)^C b^{2j}.
\label{eq:newton-error}
\end{equation}
This follows by expanding in both tail values and in
\(\lambda-\lambda_w=O(\rho b^j)\). Second derivatives in tail
values are \(O(\rho^2\lambda^{2j})\); mixed derivatives are at
most \(C(L+j+1)^C\rho\lambda^j\); the second parameter
derivative is bounded.

Split a parent kernel exactly into four affine child kernels,
then replace them by the kernels at their chosen freezing points.
Coupling both through the actual root and dividing by their
comparable widths bounds the normalized displacement by
\[
C\rho(L+j+1)^C(b^2/a)^j.
\]
Normalized supports are uniformly bounded, and their Sobolev
norms are bounded by
\(C(M(\lambda_w)+M(\lambda_{wi}))\), by
\eqref{eq:common-sobolev-bound}.
For bounded-support densities of \(H^\sigma\) norm at most \(A\),
smoothing at scale \(t\) gives
\[
\|h-\widetilde h\|_1
\leq C A^{1/(1+\sigma)}
W_1(h,\widetilde h)^{\sigma/(1+\sigma)}.
\]
Indeed the two terms are \(CAt^\sigma\) and \(CW_1/t\);
optimization proves the inequality.

Let \(P_j\) be the mixture of all depth-\(j\) kernels with
weights \(4^{-j}\). Weighted Hölder and
\eqref{eq:average-freezing} give
\begin{equation}
\|P_{j+1}-P_j\|_1
\leq C_{\lambda_0}\rho^{\sigma/(1+\sigma)}(L+j+1)^C
\left(\frac{(b^2/a)^{2\sigma}}{b^{1-d}}\right)^{j/[2(1+\sigma)]}.
\label{eq:ae-summable}
\end{equation}
By \eqref{eq:ae-exponents} this is summable. The total sum tends
to zero as \(L\to\infty\), because \(\rho=\lambda_0^L\).
The coupling estimate identifies the weak limit as the true
root law \(\nu\), so convergence also holds in \(L^1\).
In particular, if \(h_0\) is the first kernel density,
\begin{equation}
\|\nu-h_0\,d\lambda\|_{\rm TV}=o(1),\qquad
\|h_0\|_2^2\leq C_{\lambda_0}\rho^{-1},\qquad
\operatorname{supp}h_0\subset
[\lambda_0-C'\rho,\lambda_0+C'\rho].
\label{eq:density-point-kernel}
\end{equation}

Let \(G\) be the set of parameters satisfying the asserted
dimension equality. This is Borel: on compact parameter
intervals the fibres have closed graph and are upper
semicontinuous, and their Hausdorff contents can be tested by
countably many finite rational-ball covers.
Lemma~\ref{lem:root-fibres} applies to
\eqref{eq:centered-cylinders}, so \(\nu(G)=1\).
Equations \eqref{eq:density-point-kernel} and Cauchy--Schwarz imply
\[
|G\cap[\lambda_0-C'\rho,\lambda_0+C'\rho]|
\geq c_{\lambda_0}\rho
\]
for every sufficiently large \(L\).
If the complement of \(G\) had positive measure, it would have
a density point that was also a finite Lebesgue point of \(M^2\).
The preceding bound along \(\rho=\lambda_0^L\to0\) contradicts
that density property.
A countable cover by intervals satisfying \eqref{eq:ae-exponents}
proves the theorem.
\end{proof}

\begin{example}\label{ex:ae}
Fix \(\alpha=\sqrt2\). The theorem applies to almost every
\(\lambda\in(7/16,9/20)\), since \(49/256>3/16\).
For those parameters,
\[
\frac{2\log2}{\log(16/7)}-1
<\dim_{\mathrm H}(E_{\sqrt2,\lambda})<
\frac{2\log2}{\log(20/9)}-1.
\]
This assertion does not identify an individual parameter outside
the exceptional null set.
\end{example}

The Hausdorff phase transition diagram is still incomplete.
Although we have located or bounded the first and last transitions,
the distribution of zero- and positive-dimensional parameters between
them remains largely unresolved. The unit-circle case shows that
positive dimension need not persist after its first appearance;
understanding this intermediate region for other exponents is a
natural direction for further study.

\begingroup
\renewcommand{\addcontentsline}[3]{}
\section*{Acknowledgements}
\endgroup
Kan Jiang was supported by the National Natural Science Foundation of China
(No.~12671112).
Nero Ziyu Li was supported by the Additional Funding Programme for Mathematical
Sciences, delivered by EPSRC (EP/V521917/1) and the Heilbronn Institute for
Mathematical Research.
He was also supported by the EPSRC Centre for Doctoral Training in Mathematics
of Random Systems: Analysis, Modelling and Simulation (EP/S023925/1).
Lifeng Xi was supported by the National Natural Science Foundation of China
(No.~12571096).

\bibliographystyle{amsplain}
\bibliography{references}

\appendix
\section{Exact verification and certificates}\label{app:certificates}

The computations below use exact integer or rational inequalities.
We retain their verification predicates and mathematical implications
here. The certificates, verification scripts, calculation
records and reproduction instructions are available in the
\href{https://github.com/Nero-17/LEAN-Formalisation-solutions-of-fermat-cantor-equation/tree/v1.2.0}
{versioned computational companion repository}.
These are Python certificate checks; they are not Lean formalizations.

\paragraph{Polynomial exclusions.}
For a polynomial \(p\), substitute
\(\lambda=a+(b-a)t\) and write
\(p(a+(b-a)t)=\sum_{i=0}^d c_i t^i\).
Its Bernstein coefficients on \([a,b]\) are
\[
B_k=\sum_{i=0}^k c_i\frac{\binom{k}{i}}{\binom d i},
\qquad 0\leq k\leq d.
\]
The identity
\[
p(a+(b-a)t)=\sum_{k=0}^d B_k\binom d k t^k(1-t)^{d-k}
\]
shows that positive rational \(B_k\) prove strict positivity on
the entire interval. We apply this test on
\([2679/10000,67/250]\).
The following rational lower bounds for the minimum Bernstein
coefficient suffice for the inequalities used in
Theorem~\ref{thm:circle-zero-intervals}:
\[
\begin{array}{l|r}
\text{inequality}&\text{strict lower bound}\\ \hline
\text{root derivative}&3\\
\text{earlier-scale exclusion}&7/10\\
\text{first projection gap}&7/100\\
\text{nested gap below other first children}&15/100\\
\text{nested gap above other second children}&19/100\\
\text{later-scale lower endpoint}&3/100\\
\text{later-scale upper endpoint}&2/100\\
\text{first-scale lower endpoint}&6/10000\\
\text{first-scale upper endpoint}&16/10000 .
\end{array}
\]
These are downward rational bounds, not rounded tests.

\paragraph{Digital graph verification.}
We specify the exact feasibility test for the digital graph.
Write a digit as \(a_ir+b_i\), with
\((a_i,b_i)\in\{(0,0),(0,1),(1,0),(1,1)\}\).
The exclusive intervals have endpoint numerators
\[
\ell_i=(0,0),(1,1),(1,3),(3,1),\qquad
u_i=(0,2),(2,0),(2,2),(3,3)
\]
over denominator six. After a prefix of length \(j\), write its
translation as \(A_jr+B_j\), starting at zero and updating by
\(A_{j+1}=3A_j+a_j,\ B_{j+1}=3B_j+b_j\).
The condition at position \(j\) pulls back to endpoints
\[
\frac{(6A_j+\ell_{i,1})r+6B_j+\ell_{i,2}}{6\,3^j},
\qquad
\frac{(6A_j+u_{i,1})r+6B_j+u_{i,2}}{6\,3^j}.
\]
Every lower endpoint must be at most every upper endpoint.
Intersect these linear inequalities with \([1,3]\), and retain
the word precisely when this closed feasible interval meets the
open slope band. The checker enumerates all 1024 words,
then checks every edge against the supplied vertex ranks.
Every same-rank edge must be the indicated successor on a
listed simple cycle. These tests prove the graph conclusion
used in Lemma~\ref{lem:good-slopes}; discovering the graph is
not part of the verification.

\paragraph{Ordinary binary returns.}
The underlying exact recursion can be described without a grid.
Let \(K\) be the union of the three closed switch intervals
shrunk by the specified margin. Starting with \(A_0=K\) and
\(D_0=\varnothing\), set
\begin{align*}
A_{j+1}&=A_j\cup\bigcup_c(A_j+c)/3,\\
D_{j+1}&=D_j\cup\bigcup_c(D_j+c)/3\\
&\quad{}\cup\bigcup_{c\ne d}
       \bigl((A_j+c)/3\cap(A_j+d)/3\bigr),
\end{align*}
where \(c,d\in\{0,1,r,r+1\}\).
The set \(A_j\) records one return of length at most \(j\);
\(D_j\) records two incomparable returns. Induction proves
both interpretations, since the last union chooses different
first digits. Exact unions and intersections of closed rational
intervals implement the recursion. The assertion
\(B_r\subset D_{L_\alpha}\), together with the nonlinear error bound
in the proof, is the required finite certificate.

\paragraph{Weighted returns.}
The weighted-return check uses grid denominator \(N=10^{25}\), depth \(33\), and
integer component weights \(10000,20000,10000\).
Each switch component is covered by closed intervals with
integer endpoints divided by \(N\). The complete partition
contains 232304 intervals. Its destination core is the same
component shrunk by \(1/100000\), then rounded inwards.

For a digit \(c\), a current integer interval \([a,b]\) is
enclosed after one actual nonlinear step by
\[
[\,3a-\lfloor Nc\rfloor-1,\ 
   3b-\lfloor Nc\rfloor+1\,].
\]
Indeed, the remainder is nonnegative and
\[
3N\operatorname{Lip}(e)
\leq\frac{12500000000000000000000000}
          {2747555304704286477874349469}<1.
\]
For \(c=r,r+1\), the sum of this error bound and the fractional
part of \(Nc\) is exactly
\[
\frac{1176692795523008485254935882}
     {8242665914112859433623048407}<1;
\]
for \(c=0,1\), the fractional part is zero.
Thus the combined translation rounding and signed nonlinear
correction are contained in the displayed enclosure.
The same bound holds after every descendant normalization.

For an interval \(I\) and remaining depth \(h\), compute the
maximum certified antichain weight recursively: stop with the
weight of a core containing \(I\), or continue to all four
outward-rounded children and multiply their sum by \(50/57\).
Take the larger of these two choices; stopping excludes all
descendants of that node, so the selected words form an
antichain. At depth zero only stopping is permitted.
Intervals outside a fixed containing interval of
\([0,(r+1)/2]\) contribute zero. Induction proves that this
recursion is a lower bound for the available weighted return
sum. Compute it as an integer numerator over \(57^h\), without
floating-point arithmetic. On every starting partition interval
in component \(i\), the result exceeds
\(\operatorname{weight}(i)+1\) in integer weights.
Since the largest component weight is 20000, this proves the
factor \(20001/20000\) used in the theorem.

The initial level belongs to the middle switch component and
outside its shrunken endpoints. These are exact rational checks
using the printed \(P,Q\).
The verifier also checks the ternary digits, the complete
partition coverage, the inward cores, and the error constant.

\paragraph{Logarithmic sources and the infinite exponent tail.}
For each logarithmic source we verify the ternary digits, the two
strict logarithmic corner signs, and the full rational
derivative-ratio interval.
We enclose logarithms by exact range reduction and
\[
\log z=2\sum_{j=0}^{J-1}
  \frac{((z-1)/(z+1))^{2j+1}}{2j+1}+R_J,\qquad z\geq1,
\]
where
\[
0\leq R_J\leq
\frac{2((z-1)/(z+1))^{2J+1}}
 {(2J+1)(1-((z-1)/(z+1))^2)} .
\]
For \(z<1\), use \(\log z=-\log(1/z)\).
Range reduction uses exact powers of two and the same enclosed
value of \(\log2\). All endpoint rounding is outward.

The fixed-scale and tail checks test each closed parameter cell,
each Cantor rectangle, and its assigned slope band.
They compute inner corner-image intervals
with a strict inset \(10^{-12}\), verify that their translates
cover the circle, and then verify that the closed parameter
cells cover \([1,3]\).
For the tail, the integral enclosure printed in the proof
controls every \(Z\geq1180980\); testing a finite list of
exponents is not used in its place.

\paragraph{Integer sources and the continuous exponent tail.}\label{app:integer-certificates}
The logarithmic source criterion in Theorem~\ref{thm:real-powers} also applies when the reduced denominator is one. The circle case is supplied by Theorem~\ref{thm:circle-quantitative}.

Here is the finite covering argument that supplies these strict
sources. For large exponents write \(\alpha=3^m t\), \(1\leq t<3\).
Choose a Cantor endpoint \(u_0=P/3^D\), an integer \(h\), and
\[
u=u_0+3^{-m}a,\qquad a\in3^hC,\qquad m\geq D+h .
\]
The depth condition makes every such \(u\) a Cantor point.
Modulo one, the logarithmic source condition asks whether
\[
\frac{\log v-3^mt\log(1+a/(3^mu_0))}{\log3}
\]
contains the target phase
\((\alpha\log u_0-\log \alpha)/\log3\).
It suffices that strict inner corner images of admissible
rectangles cover the entire circle. Adding a constant rotation,
we check instead the phase
\[
\frac{\log v+t3^m
        \log((3^mu_0+3^h/2)/(3^mu_0+a))}{\log3}.
\]
At fixed \(m\), its corner values are affine in \(t\), so their
bounds on a closed parameter interval occur at its endpoints.
For a uniform tail in \(m\), use the exact identity
\[
Z\log\frac{Z+3^h/2}{Z+a}
       =\int_a^{3^h/2}\frac{Z}{Z+s}\,ds,\qquad Z=3^mu_0.
\]
For fixed \(a\), this is monotone in \(Z\), with direction
determined by the orientation of the integral. Its value for
every \(Z\geq Z_{\min}\) is enclosed by its value at \(Z_{\min}\)
and its limit \(3^h/2-a\). Derivative ratios have analogous
rational endpoint bounds. Thus finitely many closed coefficient
cells certify the whole infinite tail.

The first two rows below concern integer exponents; the remaining rows cover continuous exponent ranges. The exact certificate decomposition is
\[
\begin{array}{c|l}
\text{exponents}&\text{verification}\\ \hline
2\leq \alpha\leq300&\text{strict sources and the nine returns below}\\
301\leq \alpha\leq531440&\text{individual strict logarithmic sources}\\
3^{12}\leq \alpha<3^{13}&8921\text{ closed parameter cells}\\
3^{13}\leq \alpha<3^{14}&6027\text{ closed parameter cells}\\
3^{14}\leq \alpha<3^{15}&5310\text{ closed parameter cells}\\
\alpha\geq3^{15}&5862\text{ coefficient cells, }293\text{ endpoint choices}.
\end{array}
\]
For the three fixed scales \(D=9,h=3\), and the two parameter
domains \([1,2]\), \([2,3]\) are covered without gaps.
For the tail, \(D=11,h=3,u_0\geq20/243\), and
\(Z_{\min}=1180980\); the symbolic condition \(15\geq11+3\)
holds. The individual-source range ends at \(3^{12}-1\).
Thus these certificates leave no uncovered integer.
Appendix~\ref{app:certificates} specifies their data, exact
logarithm bounds, and verification predicates.

Nine small exponents use direct returns instead of admissible
slopes:
\[
\begin{array}{c|rrrrrrrrr}
\alpha&9&10&23&25&28&31&32&78&82\\ \hline
L_\alpha&21&15&16&17&18&16&20&16&17 .
\end{array}
\]
For completeness, the common transfer is as follows. In a
certified cylinder \(x=(P+b)/3^p,\ y=(Q+a)/3^q\), \(q\geq p\),
multiply the Fermat residual by \(3^{\alpha q}\) and divide by
\(2\alpha3^{\alpha(q-p)}P^{\alpha-1}\). It becomes
\[
\frac{ra+b}{2}-T+H(a,b)=0,\quad
r=\frac{Q^{\alpha-1}}{3^{\alpha(q-p)}P^{\alpha-1}},\quad
T=\frac{3^{\alpha q}-Q^\alpha-3^{\alpha(q-p)}P^\alpha}
         {2\alpha3^{\alpha(q-p)}P^{\alpha-1}}.
\]
The remainder vanishes at zero and has sup-norm Lipschitz constant
at most
\[
\frac{(Q+1)^{\alpha-1}-Q^{\alpha-1}}{2\,3^{\alpha(q-p)}P^{\alpha-1}}
+\frac{(P+1)^{\alpha-1}-P^{\alpha-1}}{2P^{\alpha-1}}.
\]
If the certificate uses a smaller rational slope \(r_0\), add
\((r-r_0)/2\) to this bound. The exact checks give \(T\in B_{r_0}\),
two incomparable returns of length at most \(L_\alpha\) into the
\(1/100000\)-interior of \(B_{r_0}\), and the strict inequality
\(3^{L_\alpha}\operatorname{Lip}(H)<1/100000\), including this slope
adjustment. Descendant normalization preserves the same error
bound. Lemma~\ref{lem:strict-returns} therefore gives
\(\dim_{\mathrm H}(E_{\alpha,1/3})\geq\log2/(L_\alpha\log3)>0\).
Together with the strict sources, this proves the integer case used in Theorem~\ref{thm:real-powers}.

\paragraph{Certificates for all real exponents.}\label{app:real-certificates}
The data and exact checkers for Theorem~\ref{thm:real-powers}
are included in release v1.2.0 of the repository cited above.
The guide \nolinkurl{docs/real-exponents.md} identifies the accepted
inputs, verification commands and previously completed verification
records; \nolinkurl{docs/paper-index.md} connects the paper's results
to their computational files. All acceptance tests use integer or rational
arithmetic, including outward enclosures of logarithms and algebraic
roots; finite-depth growth estimates play no role in the proof.

The sixteen-source denominator cover uses the following rows.
Every \(P,Q\) has ternary digits in \(\{0,2\}\). Substituting each
row in the open-interval formula in the proof and sorting the resulting
rational endpoints verifies strict overlaps covering \([1,3]\).
All listed slope bands are admissible bands or their reciprocals.
\[
\begin{array}{rr|cc}
P&Q&\ell&u\\ \hline
180&222&89/73&171/137\\
224&180&137/171&73/89\\
180&218&89/73&171/137\\
182&218&89/73&171/137\\
216&168&137/171&73/89\\
186&218&89/73&171/137\\
162&242&89/57&171/89\\
182&224&89/57&171/89\\
186&188&89/57&171/89\\
224&186&89/57&171/89\\
240&164&89/57&171/89\\
242&162&89/57&171/89\\
504&216&89/73&171/137\\
512&216&89/73&171/137\\
222&182&137/57&219/89\\
540&222&89/73&171/137\\
\end{array}
\]

For the finite range \(1<a/b<2187\), \(2\leq b\leq109\),
the checker enumerates every coprime pair. It first tests the
denominator-dependent width cover, then checks an individual source
where necessary. Each individual record gives \((a,b,P,Q,J)\) and
the slope band. With \(p,q\) the ternary depths, set
\(d=(J+ap-bq)/b\). The corner signs are exactly
\[
3^Jb^bP^a<a^b(Q+1)^b,\qquad
3^Jb^b(P+1)^a>a^bQ^b.
\]
Negative powers of three are cleared before an integer comparison.
For the large collection, range-reduced logarithms are enclosed by
the positive series for \(2\operatorname{arctanh}z\), \(0\leq z\leq1/3\),
with explicit tail bounds and outward integer rounding. A separate
implementation uses seventy terms and 200-bit fixed-point arithmetic.
It re-enumerates the entire domain and verifies the partition
\[
8{,}033{,}550=7{,}738{,}063+294{,}016+1{,}464+7.
\]
The \(1{,}464\) additional sources are also checked by direct integer
powers. The seven remaining pairs are precisely those in the
direct-return table in the proof.

For the five layers \(m=7,\ldots,11\), write \(\alpha=3^mt\),
\(1\leq t\leq3\), and use the rotated phase displayed in
Appendix~\ref{app:integer-certificates}. At each fixed endpoint
\(u_0=P/3^D\), the admissible rectangles have strict derivative
bounds throughout a closed \(t\)-cell. Their corner-image intervals
are shrunk by \(10^{-12}\). For each assigned denominator \(b\),
multiply these intervals by \(b\) and reduce modulo one. Exact
interval unions cover the whole circle, which proves coverage of
every target phase modulo \(1/b\). Different denominators may use
different endpoints, but every assigned endpoint by itself covers
its denominator's entire phase circle. Intervals from distinct
endpoints are never combined without their rotations. The checker
reconstructs the corner and derivative bounds and verifies all
closed-cell joins from \(1\) to \(3\).

Six special exponents use the following positive source cylinders
and actual integer scalings:
\[
x=3^{-k}\frac{P+s}{3^p},\qquad
y=1-3^{-m}\frac{Q+1-t}{3^q},\qquad s,t\in C.
\]
\begin{center}
\small
\begin{tabular}{c|r|r|c}
\(\alpha\)&\(P\)&\(Q\)&\((p,q,k,m)\)\\ \hline
\(24/23\)&9006915410325491654&10818090212415614952&\((40,40,110,115)\)\\
\(19/18\)&9110013456467818106&10841815821954030882&\((40,40,104,110)\)\\
\(13/12\)&9120707925586079678&10817949568504999626&\((40,40,105,114)\)\\
\(8/7\)&9450592763918769000&10918429022892060476&\((40,40,68,78)\)\\
\(11/4\)&48241792345204752&37518532384253874&\((35,35,103,284)\)\\
\(53/2\)&105672039117920552328&8521470282198373850&\((42,40,101,2680)\)\\
\end{tabular}
\end{center}

In each row \(d=m-\alpha k\in b^{-1}\mathbb Z\).
For the limiting equation \(u^\alpha-\alpha3^{-d}v=0\), divide by
\(2\alpha3^{-d-q}\). Its tangent ratio \(r_*\) at the lower corner
and its level satisfy
\[
r_*^b=3^{bd+b(q-p)}(P/3^p)^{a-b},\qquad
T=\frac{Q+1-(b/a)r_*P}{2}.
\]
The certificate encloses these quantities by rational root brackets,
selects a rational reference slope, and bounds the full variation
of both derivatives. For \(1<\alpha<2\),
\((1-z)^{\alpha-2}\leq3/2\) on the actual cylinders; for \(\alpha\geq2\)
it is at most one. Taylor's theorem therefore bounds the additional
normalized constant and \(t\)-derivative errors, respectively, by
\[
\frac{3(\alpha-1)}8\,3^{q-m},\quad
\frac{3(\alpha-1)}4\,3^{-m}
\quad(1<\alpha<2),
\]
and by the same expressions with \(3/2\) replaced by \(1\)
when \(\alpha\geq2\). The checks include the shifted actual level,
the reference-slope error and the strict derivative margin.

The seventh source, \(\alpha=8/3\), is directly on the actual curve
and is displayed after the theorem. Rational brackets for the
cube roots of the fifth powers of its two coordinate endpoints
enclose both first derivatives; multiplication by the endpoints
also encloses the residual. The two source checks use independent
root-enclosure implementations.

For all seven sources, the binary-return recursion given above
is applied to the \(\eta\)-interior of \(B_r\), using the specified
rational reference slope. It verifies the full inclusion
\(B_r\subset D_M\), not a sample of levels. Together with the
strict nonlinear error bound in the theorem, this proves that
the two incomparable returns can be repeated at every descendant.
For \(8/3\), a second interval implementation reproduces all
seventeen generations of the recursion, ending with \(289{,}323\)
one-return intervals and \(1{,}787{,}353\) two-return intervals.

\paragraph{Certificates for the uniform last-transition bound.}
\label{app:uniform-certificates}
The data and independent verification programs for
Theorem~\ref{thm:last-real} are supplied in the
\href{https://github.com/Nero-17/LEAN-Formalisation-solutions-of-fermat-cantor-equation/tree/v1.2.0/supplements/universal-last-bound}{versioned computational supplement} (release v1.2.0).
The package includes a manifest of data hashes and the accepted
verification reports. With Python~3.11 or later and its standard library, run from the repository root
\begin{verbatim}
python reproduce.py uniform-bound
\end{verbatim}
to recompute every required finite inequality and coverage check.
The command \texttt{python reproduce.py uniform-bound-audit} checks the saved reports,
data hashes, and dependency chain; it does not repeat the arithmetic.
The mathematical reductions are given here and in the theorem's proof.
These checks are a computer-assisted proof, not a Lean formalization.

Cylinder endpoints are reconstructed from finite binary words as
polynomials in \(\lambda\), with derivative bounds obtained by
differentiating those polynomials. The verifier uses outward integer
intervals with denominator \(10^{60}\). Logarithms use binary range
reduction and \(110\) terms of
\[
\log x=2\sum_{j=0}^{\infty}
\frac{1}{2j+1}\left(\frac{x-1}{x+1}\right)^{2j+1},
\]
with a geometric bound for the remainder. Exponentials use \(80\)
positive Taylor terms with a rigorous tail bound and repeated
squaring. Rational powers are enclosed by integer root comparisons
or by these logarithm and exponential bounds.
For contraction intervals, a midpoint value plus interval derivative
times radius bounds each residual. Intersecting this bound with the
natural polynomial interval enclosure remains rigorous and reduces
subdivision for cylinder endpoints near one.
For the five continuous covers, the filenames are
\begin{verbatim}
continuous_small_n_cover.jsonl
continuous_medium_n_cover.jsonl
continuous_high_n_cover.jsonl
continuous_bridge_n_cover.jsonl
continuous_upper_n_cover.jsonl
\end{verbatim}
There are \(89240\) records and \(107278\) verified subrectangles.
The independent exponent sweeps have respectively
\(48,493,1842,2017,2553\) strips.

\paragraph{The interval \(1<\alpha\leq6/5\).}\label{app:uniform-near-one}
For \(1<\alpha\leq107/100\), use the outer cylinders
\(u=1-\lambda+\lambda s\), \(v=1-\lambda t\), and choose
\begin{equation}\label{eq:near-one-window}
\max\{1,1/\lambda-1-\lambda\}<z<\frac1{\alpha(1-\lambda)}.
\end{equation}
The model \(u^\alpha=\alpha zv\) crosses its corners strictly, and its
ratio interval is \([z,z/(1-\lambda)^{\alpha-1}]\).
This lies in the first band, since
\(\alpha(1-\lambda)^{\alpha-1}>1\); indeed
\(\log \alpha>(\alpha-1)\log2\) for \(1<\alpha<2\).
For reduced \(\alpha=p/q\), the available values are
\(z=\lambda^{j/q}\), \(j\in\mathbb Z\), whose multiplicative
spacing is \(\lambda^{-1/q}\), and \(1/q\leq \alpha-1\).
The window contains a grid point if
\[
\alpha(1-\lambda)\max\{1,1/\lambda-1-\lambda\}<\lambda^{\alpha-1}.
\]
The left factor in \(\lambda\) decreases, and its value at
\((\sqrt3-1)/2\) is \(\sqrt3/2\). The worst exponent is
\(107/100\). The exact comparisons
\[
\frac{107\sqrt3}{200}<\frac{1391}{1500}
 <(9/25)^{7/100}
 <\left(\frac{\sqrt3-1}{2}\right)^{7/100}
\]
follow from \(\sqrt3<26/15\) and
\((9/25)^7>(1391/1500)^{100}\). Irrational exponents have
dense coefficient phases. For rational exponents, solving
\(qm-pk=j\) and increasing both scales gives the exact coefficient;
the endpoint transfer in the proof applies.

For \(107/100<\alpha\leq11/10\) and \(q\geq21\), the same window
works because
\[
\frac{11\sqrt3}{20}
 <\left(\frac{\sqrt3-1}{2}\right)^{1/21}
 \leq\lambda^{1/q}.
\]
For an exact check use \(\sqrt3<97/56\) and
\((1067/1120)^{21}<183/500<(\sqrt3-1)/2\).
The reduced fractions with \(q\leq20\) in this interval are
\(11/10,12/11,13/12,14/13,15/14\).
For \(12/11,13/12,15/14\), respectively, choose
\(z=\lambda^{-4/11},\lambda^{-4/12},\lambda^{-5/14}\)
in \eqref{eq:near-one-window}. Checking its two inequalities at
\((\sqrt3-1)/2\) suffices: its upper endpoint increases, while
the chosen \(z\) decreases; where the lower endpoint exceeds one,
its logarithmic rate of decrease is greater than one and hence
greater than the rate of these powers.

For \(11/10\) and \(14/13\), use cylinders \(10,11\), with hulls
\([1-\lambda,1-\lambda+\lambda^2]\) and
\([1-\lambda^2,1]\), and coefficients \(\alpha\lambda^t\) with
\(t=2/5,5/13\), respectively. After reflection and exchange of
coordinates, the ratio interval is
\begin{equation}\label{eq:near-one-second}
\left[\frac{(1-\lambda)^{\alpha-1}}{\lambda^t},
       \frac{(1-\lambda+\lambda^2)^{\alpha-1}}{\lambda^t}\right].
\end{equation}
Here \(t>\alpha-1\) and \(\alpha-1+t<1\). Both endpoints decrease in
\(\lambda\), the lower remains above one at \(1/2\), and the
upper band endpoint increases. Where \(1/\lambda-1-\lambda>1\),
the absolute logarithmic derivative of the lower ratio with respect
to \(\log\lambda\) is at most \(\alpha-1+t<1\), whereas that of
\(1/\lambda-1-\lambda\) is greater than one. Thus the band
comparisons reduce to \(\lambda=(\sqrt3-1)/2\).
The lower corner also reduces to this endpoint; for the upper
corner it suffices to check
\[
(3/4)^\alpha>\alpha(1/2)^t
\left(1-\left(\frac{\sqrt3-1}{2}\right)^2\right).
\]
The program \texttt{parent\_near\_one\_audit.py} checks all these
rational-power inequalities exactly.

For \(11/10\leq \alpha\leq6/5\), the same two cylinders work
uniformly for \(21/50\leq t\leq47/100\).
The lower ratio is worst at \((\alpha,t)=(6/5,21/50)\), the upper
at \((11/10,47/100)\), and the lower corner at
\((11/10,47/100)\). The preceding monotonicity again reduces
these comparisons to \((\sqrt3-1)/2\).
For the upper corner, subdivide the contraction interval at
\(19/50,2/5,21/50,11/25,23/50,12/25,1/2\).
On each successive interval \([a,b]\), check
\[
(1-b+b^2)^{6/5}>\frac65 b^{21/50}(1-a^2).
\]
For the first \(a\), use a rational lower enclosure of
\((\sqrt3-1)/2\). These seven strict inequalities are checked
by the same program. Since \(47/100-21/50=1/20>1/21\),
every coefficient grid with \(q\geq21\) meets the open window.
The exact small-denominator enumeration and chosen powers are
\[
\begin{array}{c|rrrrrrrrr}
\alpha&21/19&10/9&17/15&8/7&23/20&15/13&22/19&13/11&19/16\\
t&8/19&4/9&7/15&3/7&9/20&6/13&8/19&5/11&7/16
\end{array}
\]
inside this window. Three further pairs,
\((19/17,8/17),(9/8,3/8),(20/17,8/17)\), satisfy the same
corner and band comparisons individually, including the seven
upper-corner comparisons. The only remaining fractions are
\(7/6\) and \(6/5\). Their files
\texttt{n7\_6\_analytic\_cover.jsonl} and
\texttt{n6\_5\_analytic\_cover.jsonl} contain \(67\) and \(74\)
actual-source records. The independent checker verifies them on
\(121\) and \(137\) subintervals, respectively, and checks that
each union starts strictly below \((\sqrt3-1)/2\) and ends at
\(1/2\). This completes every real exponent in the stated range.

\paragraph{The unbounded exponent interval.}\label{app:uniform-tail}
Write \(\alpha=\lambda^{-m}t\), \(1\leq t<1/\lambda\).
Take a finite cylinder endpoint \(u_0>0\), an integer
\(h\in\{0,1,2\}\), and insert the genuine cylinder
\[
u=u_0+\lambda^{m-h}s,\qquad s\in C_\lambda.
\]
Its admissibility requires that the endpoint have a word of depth
at most \(m-h\); appending zero digits then gives exactly this
cylinder. On writing \(a=\lambda^{-h}s\), its logarithmic phase,
after subtracting \(\alpha\log u_0\), is
\begin{equation}\label{eq:uniform-tail-phase}
\frac{t}{u_0}(\lambda^{-m}u_0)
\log\left(1+\frac{a}{\lambda^{-m}u_0}\right)-\log v.
\end{equation}
The first derivative with respect to \(a\) is
\((t/u_0)/(1+a/(\lambda^{-m}u_0))\).
Every recorded pair of subcylinders has its whole derivative-ratio
interval inside one band of \eqref{eq:two-bands}, allowing reflection
and exchange of coordinates.

To control every finite scale above a lower bound, use
\[
Z\log\frac{Z+A}{Z+B}=\int_B^A\frac{Z}{Z+s}\,ds,
\qquad A,B\geq0.
\]
This is increasing in \(Z\) if \(A\geq B\), decreasing otherwise.
Its range for \(Z\geq Z_0\) is therefore bounded by its value
at \(Z_0\) and its limit \(A-B\). The verifier checks both in
each phase-overlap comparison. It bounds the derivative loss by
\(1+\lambda^{-h}/(\lambda^{-m}u_0)\).

Here is the phase-cover test. The recorded sources and integer
shifts give true nonempty open phase intervals \((L_i,U_i)\),
with period \(-\log\lambda\). The checker verifies strict
successive overlaps and the closing inequality
\[
U_i>L_{i+1},\qquad
U_{\mathrm{last}}>L_1-\log\lambda.
\]
These inequalities cover the entire phase circle, without requiring
the lower endpoints to be ordered. Indeed, if a phase were omitted,
each interval would lie between two consecutive translates of it.
The successive inequalities force the corresponding translate
indices to be nonincreasing; the closing inequality forces the
last index to exceed the first, a contradiction.
The cover therefore hits the phase \(\log \alpha-\alpha\log u_0\)
modulo \(-\log\lambda\). Equation~\eqref{eq:uniform-tail-phase}
then yields a strict source for \(u^\alpha=\alpha\lambda^jv\) with an
integer \(j\), to which the logarithmic endpoint transfer applies.

Each cycle is certified on an interval of coefficients \(t/u_0\).
If that interval is \([b_-,b_+]\) and the endpoint enclosure is
\([u_-,u_+]\), then \([u_+b_-,u_-b_+]\) is a common permitted
interval of reduced exponents \(t\). The verifier rounds these
products inward and checks their union covers \([1,1/\lambda]\)
throughout each contraction slab. It also reconstructs all endpoints,
checks their depths, and verifies the slab union.

For \(\lambda^{-m}\geq1000\), use
\texttt{adversarial\_phase\_small\_M1000\_D7.jsonl}.
Its \(36\) slabs contain \(33755\) sources.
Since \(((\sqrt3-1)/2)^{-6}<1000\), we have \(m\geq7\).
The endpoint-depth checks require at most \(7,6,5\) digits for
\(h=0,1,2\), respectively, so insertion is legal for every scale.
The independent replay is \texttt{parent\_uniform\_phase\_audit.py}.

If \(\lambda^{-m}<1000\) and \(\alpha\geq500\), then
\(m\in\{6,7,8,9\}\). Indeed \(m\leq5\) would give
\(\alpha<\lambda^{-6}\leq((\sqrt3-1)/2)^{-6}<500\), whereas
\(m\geq10\) would give \(\lambda^{-m}\geq2^{10}>1000\).
For each remaining \(m\), it suffices to cover
\[
\left(1000^{-1/m},500^{-1/(m+1)}\right)
\cap\left[\frac{\sqrt3-1}{2},\frac12\right].
\]
On a rational contraction slab \([l,u]\), use the enlarged ranges
\[
\lambda^{-m}\geq\lfloor u^{-m}\rfloor,
\qquad \max\{1,500l^m\}\leq t\leq1/l.
\]
The file \texttt{adversarial\_finite\_shells\_N500.jsonl}
contains \(75\) such slabs and \(22859\) sources, each with
endpoint depth at most \(m-h\). The program
\texttt{parent\_finite\_shell\_audit.py} checks the same phase,
derivative, and insertion conditions, and the exact union of slabs
for each \(m\). Initial endpoints are rounded down.
Together with the full-scale certificate this covers all
\(\alpha\geq500\), without a rationality assumption.

\end{document}